\documentclass[11pt,a4paper,reqno,draft]{amsart}
\usepackage[T1]{fontenc}
\usepackage{mathrsfs}
\usepackage{amsfonts,amssymb,amsmath,amsgen,amsthm}
\usepackage{color}
\usepackage{amsbsy}
\usepackage{mathrsfs}
\usepackage{bbm}
\usepackage{hyperref}
\usepackage{titletoc}
\usepackage{enumerate}
\usepackage{subfigure}
\usepackage[subfigure]{ccaption}
\usepackage[all]{xy}
\usepackage{pdfsync}
\newtheorem{theorem}{Theorem}[section] 
\newtheorem{assumption}{Assumption}[section]
\newtheorem{lemma}[theorem]{Lemma} 

\newtheorem{proposition}[theorem]{Proposition} 
\newtheorem{corollary}[theorem]{Corollary} 
 
\theoremstyle{remark}
\newtheorem{remark}[theorem]{\it \bf{Remark}\/}
 
\numberwithin{equation}{section}
\catcode`@=11 
\def\section{\@startsection{section}{1}%
  \z@{1.5\linespacing\@plus\linespacing}{.5\linespacing}%
  {\normalfont\bfseries\large\centering}}
\catcode`@=12
\newcommand{\be}{\begin{equation}}
\newcommand{\ee}{\end{equation}}
\newcommand{\bea}{\begin{eqnarray}} 
  \newcommand{\eea}{\end{eqnarray}} 
\newcommand{\bee}{\begin{eqnarray*}}
\newcommand{\eee}{\end{eqnarray*}}

\def\pa{\partial}
\def\na{\nabla}

\def\NN{\mathbb{N}}
\def\RR{\mathbb{R}}
\def\ZZ{\mathbb{Z}}

\def\II{\mathbb{I}}

\def\ni{\noindent} 
\def\bs{\bigskip}

\def\eps{\varepsilon}
\def\fref#1{{\rm (\ref{#1})}}

\def\calH{{\mathcal H}}
\def\calL{{\mathcal L}}
\def\calC{{\mathcal C}}

\def\calE{{\mathcal E}}

\def\calF{{\mathcal F}}

\def\calS{{\mathcal S}}
\def\calJ{{\mathcal J}}
\def\calK{{\mathcal K}}

\def\Hunper{{\rm H}^1_{per}}

\def\H{{\rm H}}  
\catcode`@=11
\def\supess{\mathop{\operator@font Sup\,ess}}
\catcode`@=12
\DeclareMathOperator{\Tr}{Tr}

\definecolor{myorange}{cmyk}{0.,0.50,1.0,0.15}

\def\un{{\mathbbmss{1}}} 
\author[F. M\'ehats]{Florian M\'ehats}
\address[F. M\'ehats]{IRMAR, Universit\'e de Rennes 1 and IPSO, INRIA Rennes, France}
\email{florian.mehats@univ-rennes1.fr}
\author[O. Pinaud]{Olivier Pinaud}
\address[O. Pinaud]{Department of Mathematics, Colorado State University, Fort Collins, CO, USA}
\email{pinaud@math.colostate.edu}
\begin{document}

\begin{abstract}
This work is devoted to the analysis of the quantum Liouville-BGK equation. This equation arises in the work of Degond and Ringhofer on the derivation of quantum hydrodynamical models from first principles. Their theory consists in transposing to the quantum setting the closure strategy by entropy minimization used for kinetic equations. The starting point is the quantum Liouville-BGK equation, where the collision term is defined via a so-called quantum local equilibrium, defined as a minimizer of the quantum free energy under a local density constraint. We then address three related problems: we prove new results about the regularity of these quantum equilibria; we prove that the quantum Liouville-BGK equation admits a classical solution; and we investigate the long-time behavior of the solutions. The core of the proofs is based on a fine analysis of the properties of the minimizers of the free energy.
\end{abstract}


\title{The quantum Liouville-BGK equation and the moment problem}
\maketitle


\sloppy
\section{Introduction}

Our motivation in this work is to pursue the development of the mathematical foundations of the formal theory introduced by Degond and Ringhofer in \cite{DR} on the derivation of quantum hydrodynamical models from first principles. Such models provide a reduced description of the dynamics of many-particles systems, and are therefore of great interest for semiconductor applications for instance. The main idea behind Degond and Ringhofer's work is to transpose to quantum systems the moment closure by entropy minimization Levermore used for kinetic equations \cite{levermore}. This is by nature a non-pertubative approach, which is in great contrast with standard techniques consisting in adding (small) quantum corrections to classical fluid models. While the theory is now relatively well understood at a formal level, many mathematical questions remain unsolved. 

The starting point of the theory is the quantum Liouville equation with collision term
  \be \label{liouville}
i \hbar \partial_t \varrho=[H, \varrho]+i \hbar Q(\varrho),
\ee
where $\varrho$ is the density operator (a self-adjoint nonnegative trace class operator), $H$ is a given Hamiltonian, $[\cdot,\cdot]$ denotes the commutator between two operators, and $Q$ is a collision operator that will be  discussed later on. As in the kinetic case, an infinite cascade of equations for the (local) moments of $\varrho$ can be derived from \fref{liouville}, and this cascade cannot be closed since higher order moments do not only depend on lower order moments. Note that the definition of the local moments of $\varrho$ is unclear at this point. A good place to start with is to think of the moments of the associated Wigner function with respect to the momentum variable, which then yields the local particle density as the first moment, the local current density as the second and so on. 

By analogy with the classical case, Degond and Ringhofer then introduced a so-called \textit{quantum local equilibrium} (in the statistical sense) in order to close the infinite cascade. Since the system is driven to the equilibrium by the collision process, this local equilibrium is naturally related to $Q$. Even though $Q$ plays a central role in the dynamics at the small scale, devising its appropriate form for a problem of interest remains mostly an open question, and $Q$ is often obtained by analogy to classical collision operators. With the perspective of deriving macroscopic quantum models, it is argued that only a few properties of $Q$ and not its exact form are needed in order to define the quantum local equilibrium and the limiting quantum fluid equations. These properties are the collision invariants and the kernel of $Q$. In Degond-Ringhofer's approach, the collision invariants are the first few local moments of $\varrho$, say $N$ of them.  The kernel is defined with the entropy principle: the Liouville equation should dissipate the entropy (here the free energy, which is actually a relative entropy) until an equilibrium is attained; it is therefore expected that operators in the kernel of $Q$ are also minimizers of the free energy. The quantum local equilibria are then formally obtained as follows: they are minimizers of the quantum free energy
$$
F(u)= T \Tr (\beta(u)) + \Tr (H u),
$$ 
where $T$ is the temperature (we will set $\hbar=T=1$ for simplicity), $\beta$ is an entropy function, and $\Tr$ denotes operator trace, under the constraints that the $N$ first local moments of $u$ (i.e. the collision invariants) are equal to those of $\varrho$, the solution to the Liouville equation. Depending on the number of moments used in the closure procedure, several quantum macroscopic models are then derived: Quantum Euler, Quantum Energy Transport, Quantum Navier-Stokes, or Quantum Drift-Diffusion in the diffusive regime, we refer to \cite{brull-mehats,QDD-JCP,isotherme,QHD-CMS,DGMRlivre,QET,QHD-review,gambajung,jungel-matthes,jungel-matthes-milisic} for more details about these models and other references on quantum hydrodynamics.

Degond and Ringhofer's theory raises many interesting mathematical questions. The first one to consider is the construction of the minimizers of the free energy, as these latter define the local equilibria that subsequently lead to the fluid equations. The main difficulty lies in the fact that it is an infinite dimensional optimization problem with \textit{local} constraints. In order to be more explicit, let us consider the first moment only, namely the particle density. If $\varrho$ is a density operator with integral kernel $\rho(x,y)$, then the local density is defined formally by $\rho(x,x)$ (other equivalent definitions are given further). The problem is then to minimize $F$ under the constraint that the local density of the minimizers is a given function, say $n(x)$. In the sequel, we will refer to this problem as the \textit{moment problem}. Partial answers were given in \cite{MP-JSP, MP-KRM}, which, up to our knowledge, are the only results on this matter. 

In \cite{MP-KRM}, it is proved that the moment problem admits a unique solution when the first two moments are prescribed in an appropriate functional setting; the allowed configurations are fairly general, multidimensional, in bounded domains or in the whole space, for different classes of entropy functions (e.g. Boltzmann or Fermi-Dirac), and for linear or non-linear Hamiltonians (e.g. including Poisson or Hartree-Fock non-linear interactions). A rigorous construction of the minimizers when the third moment (i.e. the local energy) is prescribed is an open problem. 

The next question is to characterize the minimizers. As in the classical case where the minimizers of the Boltzmann entropy are Maxwellians, it is shown formally in Degond-Ringhofer's theory that the minimizers of the quantum free energy (for the Boltzmann entropy) are the so-called \textit{quantum Maxwellians}. They have the form $\exp(- \widetilde{H})$, where the exponential is taken in the sense of functional calculus and $\widetilde{H}$ is some Hamiltonian to be defined later on. Justifying such a formula is actually a difficult task, much harder than proving the existence of the minimizers. The question is addressed in \cite{MP-JSP} where the quantum Maxwellians are properly defined in a one-dimensional setting, for the moment problem with density constraint only. The one-dimensional setting is not incidental: the obtained quantum Maxwellians have the form $\exp(-(H+A))$, where $A(x)$ is the chemical potential, the Lagrange parameter associated with the local density constraint; when $n$ is in the Sobolev space $\H^1$ (with periodic boundary conditions), then $A$ is a distribution in $\H^{-1}$, and this is an optimal result. As a consequence, the Hamiltonian $H+A$ can readily be defined in the sense of quadratic forms in the one-dimensional case, but the definition is not so obvious in a multidimensional setting for such low regularity potentials. A longstanding question was then to understand how some additional regularity on the prescribed density $n$ can be transferred to $A$. We answer this question in the present work in the one-dimensional setting, and show in particular that $A$ is a function in $L^r$ provided $n(x)$ is in the Sobolev space $W^{2,r}$ (plus another condition). This result opens new perspectives for the characterization of the minimizer in the multidimensional case (that will be addressed elsewhere), as we have now a technique to construct sufficiently regular potentials $A$ in order to define the Hamiltonian $H+A$.  

Another interesting mathematical question is to define the quantum evolution \fref{liouville}. It is an important problem since all quantum fluid models are obtained as approximations of \fref{liouville}. With the persective of validating these models and investigating their accuracy, we then need to be able to construct solutions to \fref{liouville}. As we have already seen, the key point is the structure of the collision operator $Q$. As in Degond-Ringhofer's work, we will consider a collision operator of BGK type \cite{BGK}, that is 
$$
Q(\varrho)=\frac{1}{\tau} \big(\varrho_e[\varrho]-\varrho \big),
$$
where $\tau$ is a relaxation time, and $\varrho_e[\varrho]$ denotes the quantum local equilibrium obtained by resolution of the moment problem. We will consider the moment problem with just the density constraint since it is the only case where the minimizers were properly characterized. We then suppose that $\varrho_e[\varrho]$ is a minimizer of $F$ under the local constraint $n[\varrho_e]=n[\varrho]$, where $n[\varrho_e]$ and $n[\varrho]$ denote the local densities associated to $\varrho_e$ and $\varrho$, respectively. The main mathematical difficulty is then to understand the map $\varrho \mapsto \varrho_e[\varrho]$. It is non-linear and non-local, and can essentially be seen as an inverse problem: given a density $n[\varrho(t)]$ in an appropriate space, (i) can we construct a density operator $\varrho_e$, minimizing the free energy, which yields the local density $n[\varrho(t)]$ at each time $t>0$, and (ii) how does this minimizer depend on $n[\varrho]$ (and therefore on $\varrho$) in terms of regularity and continuity? We answer these questions and show in particular that the map $\varrho \mapsto \varrho_e[\varrho]$ is of H\"older regularity in some functional setting. This then allows us to construct a solution to \fref{liouville}. Note that the question of uniqueness of this solution remains an open problem.

The last problem we address in this work is the long time convergence to the equilibrium. As mentioned above, \fref{liouville} dissipates the free energy, and we show that the solutions converge for long time to minimizers of the free energy with a \textit{global density constraint}. We do not address the question of the rate of convergence, which may be a difficult problem even in the kinetic case.

We will work with the same setting as \cite{MP-JSP}, since it is the one in which the minimizer was characterized. The entropy function is then the Boltzmann entropy, and our domain is one-dimensional with periodic boundary conditions. These latter conditions allow us to dismiss potential boundary effects and to focus on the intrinsic difficulties of the problem. The one-dimensional limitation was addressed above, and we believe the techniques developed in this paper are sufficiently general to allow us to extend in the future the results to a multidimensional setting and to other entropy functions. 

To end this introduction, we would like to mention that many problems remain unsolved in Degond-Ringhofer's theory: characterization of the minimizers for several constraints in several dimensions, definition of the quantum evolution for such local equilibria, analysis of the fluid models, rigorous semi-classical and diffusion limits, etc...

The paper is structured as follows: in Section \ref{prelim}, we introduce  the functional setting and recall the main results of \cite{MP-JSP} on the moment problem; section \ref{main} is devoted to the main results of this work (Theorems \ref{regA}, \ref{thliou} and \ref{thliou2}) and an outline of their proofs. The proof of Theorem \ref{regA} is given in Section \ref{proofregA}, and additional important results on the moment problem are given in Section \ref{further}. Section \ref{proofthliou} is devoted to the proof of Theorem \ref{thliou} and Section \ref{proofthliou2} to the one of Theorem \ref{thliou2}. The proofs of some technical lemmas are given in Section \ref{secother}, and Section \ref{appen} is an appendix gathering some results of \cite{MP-JSP}.\\

\section{Preliminaries} \label{prelim}
We start with the functional framework and then recall the main results of \cite{MP-JSP} on the moment problem.

\subsection{Functional framework} \label{sect2}
We work with a one dimensional physical space and suppose that the particles are confined in the interval $[0,1]$ with periodic boundary conditions. We will denote by $L^r(0,1) \equiv L^r$, $r\in [1,\infty]$, the usual Lebesgue spaces of complex-valued functions, and by $W^{k,r}(0,1) \equiv W^{k,r}$, the standard Sobolev spaces. We will use the notations $\H^k=W^{k,2}$, and $(\cdot,\cdot)$ for the Hermitian product on $L^2$ with the convention $(f,g)=\int_0^1 \overline{f} g dx$. We then consider the Hamiltonian $$H=-\frac{d^2}{dx^2}$$ on the space $L^2$, equipped with the domain
\be D(H)=\left\{u\in \H^2(0,1):\,u(0)=u(1),\,\frac{du}{dx}(0)=\frac{du}{dx}(1)\right\}.\label{domainH} \ee
Note that we considered a free Hamiltonian for simplicity of the exposition, and that the results obtained in this work can straightforwardly be generalized to Hamiltonians with sufficiently regular potentials. The domain of the quadratic form associated to the Hamiltonian is
$$\Hunper=\left\{u\in \H^1(0,1): \,u(0)=u(1)\right\}.$$
Its dual space will be denoted $\H^{-1}_{per}$. Remark that one has the following identification, that will consistently be used in the paper:
\be
\forall u,v \in \Hunper,\qquad (\sqrt{H}u,\sqrt{H}v)=\left(\frac{du}{dx},\frac{dv}{dx}\right),\quad \|\sqrt{H}u\|_{L^2}=\left\|\frac{du}{dx}\right\|_{L^2}.
\label{ident}
\ee
We will use the notations $\nabla=d/dx$ and $d^2/dx^2=\Delta$ for brevity. We shall denote by $\calL(L^2)$ the space of bounded operators on $L^2(0,1)$,  by $\calK$ the space of compact operators on $L^2(0,1)$, by $\calJ_1 \equiv \calJ_1(L^2)$ the space of trace class operators on $L^2(0,1)$  and by $\calJ_2 \equiv\calJ_2(L^2)$ the space of Hilbert-Schmidt operators on $L^2(0,1)$. The inner product in $\calJ_2$ is defined by $(A,B)_{\calJ_2}= \Tr (A^* B)$. More generally, $\calJ_r$ will denote the Schatten space of order $r$. See e.g. \cite{RS-80-I,Simon-trace} for more details about these spaces. 

A density operator is defined as a nonnegative trace class, self-adjoint operator on $L^2(0,1)$. For $|\varrho|=\sqrt{\varrho^* \varrho}$, we introduce the following space:
$$\calE=\left\{\varrho\in \calJ_1,\mbox{ such that } \overline{\sqrt{H}|\varrho|\sqrt{H}}\in \calJ_1\right\},$$
where $\overline{\sqrt{H}|\varrho|\sqrt{H}}$ denotes the extension of the operator $\sqrt{H}|\varrho|\sqrt{H}$ to $L^2(0,1)$. We define in the same way
$$\calH=\left\{\varrho\in \calJ_1,\mbox{ such that } \overline{H |\varrho| H}\in \calJ_1\right\}. $$
We will drop the extension sign in the sequel for simplicity. The spaces $\calE$  and $\calH$ are Banach spaces when endowed with the norms
$$\|\varrho\|_{\calE}=\Tr |\varrho|+\Tr\big(\sqrt{H}|\varrho|\sqrt{H}\big), \qquad \|\varrho\|_{\calH}=\Tr |\varrho|+\Tr \big(H |\varrho| H\big),$$
where $\Tr$ denotes operator trace. The energy space will be the following closed convex subspace of $\calE$:
$$\calE_+=\left\{\varrho\in \calE:\, \varrho\geq 0\right\}.$$
Note that operators in $\calE_+$ are automatically self-adjoint since they are bounded and positive on the complex Hilbert space $L^2$. For any $\varrho\in \calJ_1$ with $\varrho=\varrho^*$, one can associate a real-valued local density $n[\varrho](x)$, formally defined by $n[\varrho](x)=\rho(x,x),$ where $\rho$ is the integral kernel of $\varrho$. The density $n[\varrho]$ can be in fact identified uniquely by the following weak formulation:
$$\forall \phi\in L^\infty(0,1),\quad \Tr \big(\Phi \varrho \big)=\int_0^1\phi(x)n[\varrho](x)dx,
$$
where, in the left-hand side, $\Phi$ denotes the multiplication operator by $\phi$ and belongs to $\calL(L^2(0,1))$. We list below some important relations involving the density $n[\varrho]$ and the spectral decomposition of $\varrho$ (counting multiplicity) that we denote by $(\varrho_k,\phi_k)_{k\in \NN^*}$. We have first
\begin{align}
&n[\varrho](x)=\sum_{k \in \NN^*}\rho_k|\phi_k(x)|^2,\qquad \|n[\varrho]\|_{L^1}\leq \sum_{k\in \NN^*} |\rho_k|=\Tr |\varrho|,\label{a5}\\
&\Tr \big(\sqrt{H}|\varrho|\sqrt{H}\big)=\|\sqrt{H}\sqrt{|\varrho|}\|^2_{\calJ_2}=\sum_{k \in \NN^*}|\rho_k|\left\|\nabla \phi_k \right\|_{L^2}^2,\label{a6}
\end{align}
where we assume $\varrho \in \calE$ in the last line. The Cauchy-Schwarz inequality then yields
\begin{align} \label{estW11}
&\left\| \nabla n[\varrho]\right\|_{L^1}\leq 2(\Tr |\varrho|)^{1/2}\left(\Tr \sqrt{H}|\varrho|\sqrt{H}\right)^{1/2}\leq \|\varrho\|_{\calE}\\
\label{sqrt}
&\left\|\nabla \sqrt{n[\varrho]}\right\|_{L^2}\leq \left(\Tr \sqrt{H}\varrho\sqrt{H}\right)^{1/2}, \qquad \textrm{if } \varrho \geq 0.
\end{align}
Throughout the paper, $C$ will denote a generic constant that might differ from line to line.
\subsection{The moment problem} For $\beta(x)=x \log x-x$ the Boltzmann entropy, we introduce the free energy $F$ defined on  $\calE_+$ as
\be
\label{F}
F(\varrho)= \Tr \big(\beta(\varrho)\big)+\Tr \big(\sqrt{H}\varrho\sqrt{H}\big).
\ee
The following theorem, proved in \cite{MP-JSP}, shows that for an appropriate given local density $n(x)$, $F(\varrho)$ admits a unique minimizer in $\calE_+$ under the constraint $n[\varrho]=n$. Moreover, this minimizer is characterized, and defined as a so-called  quantum Maxwellian.

\begin{theorem}  [\cite{MP-JSP} Solution to the moment problem]
\label{theo1}
Consider a density $n\in \Hunper$ such that $n(x)\geq \underline{n}>0$ on $[0,1]$. Then, the following minimization problem with constraint:
$$\min F(\varrho)\mbox{ for $\varrho\in \calE_+$ such that }n[\varrho]=n,
$$
where $F$ is defined by \fref{F}, is attained for a unique density operator $\varrho[n]$, which has the following characterization. We have
$$\varrho[n]=\exp\left(-(H+A)\right),
$$
where $A$ belongs to the dual space $\H^{-1}_{per}$ of $\Hunper$, is real-valued, and the operator $H+A$ is taken in the sense of the associated quadratic form
\be
\label{qu}
Q_{A}(\varphi,\varphi)=\left\|\nabla \varphi \right\|_{L^2}^2+(A, |\varphi|^2)_{\H^{-1}_{per},\Hunper}.
\ee
Moreover, the operator $\varrho[n]$ has a full rank, that is all of its eigenvalues $(\rho_p)_{p \in \NN^*}$ are strictly positive.
\end{theorem}

It is shown in addition in \cite{MP-JSP} that the chemical potential $A$ of Theorem \ref{theo1} is defined by the relation
\be
\left(A,\psi\right)_{\H^{-1}_{per},\Hunper}=-\Tr\left(\frac{\psi}{n} (\varrho\log \varrho)\right)-\sum_{p\in\NN^*}\left((\sqrt{H}\sqrt{\varrho})\phi_p,\sqrt{H}\left(\frac{\psi}{n}\sqrt{\varrho}\phi_p\right)\right)_{L^2}\label{defAA},
\ee
where $(\phi_p)_{p \in \NN^*}$ is the set of eigenfunctions of the minimizer $\varrho[n]$. For the analysis of the Liouville equation, some estimates on $\varrho[n]$ and the Lagrange parameter $A$ in terms of the data of the minimization problem, namely the density $n$, are necessary. This question is addressed in the following proposition, proved in Section \ref{proofseveral}.

\begin{proposition} (Estimates for the solution to the moment problem)\label{severalestimates} Under the conditions of Theorem \ref{theo1}, the minimizer $\varrho[n]$ and the chemical potential $A$ verify the estimates:
\begin{align} \label{estimsolmom1}
&\Tr \big(|\varrho[n] \log \varrho[n]|\big)+ \| \varrho[n] \|_{\calE} \leq C \mathfrak{H}_0(n) \\\label{estimsolmom2}
&\|A\|_{\H^{-1}_{per}} \leq C \mathfrak{H}_1(n),
\end{align} 
where the functions $\mathfrak{H}_0$ and $\mathfrak{H}_1$ are defined by
\begin{align*}
&\mathfrak{H}_0(n)=1+\beta(\|n\|_{L^1})+\| \sqrt{n}\|^2_{\H^1}\\
&\mathfrak{H}_1(n)=\left(1+\|\sqrt{n}\|_{\H^1}/\sqrt{\underline{n}} \right)\mathfrak{H}_0(n)/\underline{n},
\end{align*}
and $C$ is a constant independent of $n$.
\end{proposition}

\section{Main results} \label{main}

We state in this section the main results of the paper. The first theorem concerns the regularity of the chemical potential $A$.
\begin{theorem} [Regularity for the chemical potential] \label{regA}
Consider a density $n\in \Hunper$ such that $n\geq \underline{n}>0$ on $[0,1]$. Let $A$ be the chemical potential associated with the minimizer $\varrho[n]$ of Theorem \ref{theo1}. Then $\Delta n \in W^{k,r}(0,1)$ implies $A \in W^{k,r}(0,1)$, for $r \in [1,\infty]$. When $k=0$ and $r=2$, we have in particular the estimate 
$$
\|A\|_{L^2} \leq \frac{C}{\underline{n}} \left(\mathfrak{H}_0(n)\left(1+\frac{1}{\underline{n}} \left(\| \Delta n \|_{L^2}+\mathfrak{H}_0(n)\right)\right)+\exp\left( C (\mathfrak{H}_1(n))^4\right) \right),
$$
where $C$ is a positive constant independent of $n$ and the functions $\mathfrak{H}_0$ and $\mathfrak{H}_1$ are defined in Proposition \ref{severalestimates}.
\end{theorem}

As mentioned in the introduction, Theorem \ref{regA} opens new perspectives for the characterization of the solutions to the moment problem in a multidimensional setting. Indeed, adapting the method of proof, we believe it is possible to construct a potential $A$ with sufficient regularity to define the Hamiltonian $H+A$ in dimensions greater than one.

Besides, it is shown formally in \cite{isotherme}, that in various asymptotic limits (semi-classical, and low-temperature), the potential $A$ behaves like the Bohm potential $\Delta \sqrt{n} /\sqrt{n}$. This means that the conditions on $n$ of Theorem \ref{regA} might not be optimal (we have two separate assumptions on $\Delta n$ and $n$, and not just one on $\Delta \sqrt{n} /\sqrt{n}$), but also that indeed some information about the second order derivatives of $n$ is required for $A$ to be a function and not a distribution.

Moreover, an important assumption in Theorem \ref{theo1} is that the density $n$ defining the constraint must be uniformly bounded from below by a positive constant. This is a sufficient condition of solvability, and we do not know at this point if the condition is necessary. In the Liouville equation, the moment problem has to be solved at each time in order to define $\varrho_e[\varrho(t)]$. This then requires the local density associated to $\varrho(t)$ to be uniformly bounded from below by a positive constant. Assuming this condition holds for the initial density operator, the lower bound can be propagated to the time dependent solution if the initial density operator has the form given below.

\begin{assumption}
\label{ass1} The initial condition is a density operator $\varrho^0$ that belongs to $\calH$, which can be decomposed into the sum of an operator function of the Hamiltonian $H$ and a perturbation, that is
$$
\varrho^0=f(H)+\delta \rho.
$$
We suppose moreover that $f(H) \in \calE_+$, that $\delta \rho \in \calE$, and that there exists a constant $\gamma>0$ such that
$$
n[f(H)](x) \geq \gamma \quad  \mbox{for all }x\in[0,1] \qquad \textrm{and} \qquad \|\delta \rho\|_{\calE} <\frac\gamma 2. 
$$
\end{assumption}

Assumption \ref{ass1} implies that that $n[\varrho^0]$ is uniformly bounded from below. Indeed, by linearity of the trace and denoting $\eps=\frac{2 \|\delta \rho\|_{\calE}}{\gamma}\in (0,1)$, we have
$$
n[\varrho^0]=n[f(H)]+n[\delta \rho] \geq (1-\eps)\gamma>0, \qquad \mbox{on } [0,1],
$$
since, according to \fref{estW11} and a Sobolev embedding,
$$
\|n[\delta \varrho]\|_{L^\infty} \leq \|n[\delta \rho]\|_{W^{1,1}} \leq 2 \| \delta \rho \|_{\calE}.
$$

We will construct classical solutions to the quantum Liouville-BGK equation that are defined as follows. Denote by $\calL(t)$ the $\calC^0$ group $t\mapsto \calL(t) \varrho=e^{-iHt} \varrho e^{iHt}$, which is an isometry on $\calJ_1$, $\calE$ and $\calH$. The group $\calL$ preserves Hermiticity and positivity. Its infinitesimal generator $L_0$ is given by, see \cite[Chapter 5, Th. 2]{dlen5},
\begin{align*}
&D(L_0)=\{ \varrho \in \calJ_1 \textrm{ such that } \varrho D(H)\subset D(H),\; H \varrho-\varrho H \textrm{ is an operator defined on }\\
&  \quad \qquad \qquad D(H) \textrm{ that can be extended to } L^2(0,1) \textrm{ to an operator } \overline{H \varrho-\varrho H} \in \calJ_1 \},\\
&L_0(\varrho)=-i (\overline{H \varrho-\varrho H}).
\end{align*}

A classical solution to the quantum Liouville-BGK equation is then a density operator 
$$\varrho \in \calC^0([0,T], D(L_0)), \qquad \textrm{with} \qquad \partial_t \varrho \in \calC^0([0,T], \calJ_1),$$
which verifies
\be \label{liouville2}
\partial_t \varrho=L_0(\varrho)-\frac{1}{\tau}(\varrho-\varrho_e[\varrho]), \quad \textrm{on } [0,T], \qquad \varrho(t=0)=\varrho^0,
\ee
where $\varrho_e[\varrho]$ is the (unique) solution to the moment problem with constraint $n[\varrho]$. 

Our second result is the following:
\begin{theorem} (Existence result)\label{thliou}
Under Assumption \ref{ass1}, there exists a classical solution to the quantum Liouville-BGK equation \fref{liouville2} such that, for any $T>0$,
$$\varrho \in \calC^0([0,T],\calH) \qquad  \textrm{and} \qquad \partial_t \varrho \in \calC^0([0,T],\calJ_1).$$
\end{theorem}

It can be noticed that we require the regularity $\varrho^0 \in \calH$ and not just $\varrho^0 \in D(L_0) \supset \calH $, and that this regularity is propagated to the solution $\varrho$. This is related to the resolution of the moment problem, in particular to the regularity of the chemical potential $A$. The space $\calH$ seems to be a fairly natural framework to define the local equilibrium $\varrho_e[\varrho(t)]$. Indeed, as mentioned before, if we want to work with a potential $A$ as a function in $L^r$ and not as a distribution, some information about the second derivative of $n$ is required.  This will be confirmed in the representation formula of Proposition \ref{expA}. Since the fact that $\varrho(t) \in \calE_+$ only allows us to control the first derivative of $n$, more regularity on $\varrho(t)$ is needed. We will see in Lemma \ref{H2n} that $\varrho(t) \in \calH$ implies that $\Delta n[\varrho(t)] \in L^2$, and it is therefore natural to consider density operators in $\calH$. It does not seem obvious to work with less regular $\varrho$ while still preserving the condition $\Delta n[\varrho]  \in L^2$.


Our last result concerns the long-time behavior of the solutions to \fref{liouville2}. The BGK collision operator is precisely designed so as to dissipate the free energy $F$, it is therefore expected for the solutions to converge to a minimizer of $F$. This is what is shown in Theorem \ref{thliou2}, with the remark that the density constraint is now a \textit{global} constraint. 

\begin{theorem} \label{thliou2} (Convergence to the equilibrium) Denote by $\varrho_g$ the unique solution to the following minimization problem with global density constraint,
$$\min F(\varrho)\mbox{ for $\varrho\in \calE_+$ such that }\|n[\varrho]\|_{L^1}=\|n[\varrho^0]\|_{L^1}.
$$
Above, we recall that $\varrho^0$ is the initial density operator, satisfying Assumption \ref{ass1}.
Then, there exists $\eps>0$ such that $F(\varrho^0)-F(\varrho_g)<\eps$ implies that any solution $\varrho(t)$ to the quantum Liouville-BGK equation with $\varrho(0)=\varrho^0$ converges to $\varrho_g$ as $t \to \infty$ strongly in $\calJ_1$.
\end{theorem}

We would like to emphasize that the condition $F(\varrho^0)-F(\varrho_g)<\eps$ is only used in order to ensure that the local density $n[\varrho(t)]$ remains bounded from below by a positive constant at all times. 

\begin{remark} By Lemma \ref{propentropie}, there exists $\tilde \eps$ such that $\|\varrho^0-\varrho_g\|_{\calE}< \tilde \eps$ yields $F(\varrho^0)-F(\varrho_g)<\eps$. Moreover, we remark that $\varrho_g$ takes the form $Ce^{-H}=f(H)$ and that its density $n[\varrho_g]\equiv n_0$ is a strictly positive constant function. Therefore, it is clear that any $\varrho^0\in \calH$ satisfying $\|\varrho^0-\varrho_g\|_{\calE}<\min(\frac{n_0}{2},\tilde \eps)$ satisfies both Assumption \ref{ass1} and the assumption of Theorem \ref{thliou2}.
\end{remark}

\medskip
\ni
\textit{Outline of the proofs.} The main ingredient of the proof of Theorem \ref{regA} is the representation formula of Proposition \ref{expA} further. Starting from the weak definition of $A$ given in \fref{defAA}, the term $\Delta n$ is exhibited after well-chosen algebraic manipulations and integration by parts. The fact that $A \in L^1$ when $\Delta n \in L^1$ then allows us to obtain a fine characterization of the spectral elements of the minimizer $\varrho[n]$ and to obtain more regularity on $A$ (as well as on $\varrho[n]$, which is important for the Liouville equation) by bootstrapping.

Regarding Theorem \ref{thliou}, we start by constructing mild solutions that satisfy the integral equation
\be \label{mild0}
\varrho(t)=e^{-\frac{t}{\tau}}\calL(t) \varrho^0+1/\tau\int_0^t e^{-\frac{t-s}{\tau}}\calL(t-s) \varrho_e[\varrho(s)] ds.
\ee
The main difficulty is naturally to handle the non-linear non-local map $\varrho \mapsto \varrho_e[\varrho]$. We were not able to show that it is Lipschitz continuous and therefore could not apply standard fixed point theorems in Banach spaces. We then use a compactness method similar to the one of the proof of the Cauchy-Peano Theorem, which explains the lack of a uniqueness result. We define a sequence $(\varrho_k)_{k \in \NN}$ of linear solutions to \fref{mild0} and derive uniform estimates in $k$. A crucial one is a \textit{sublinear} estimate in the space $\calH$ that allows us to obtain global-in-time existence results. This estimate hinges upon the characterization of the minimizer as a quantum Maxwellian and the fact that $A \in L^2$. It then remains to pass to the limit in \fref{mild0}, which amounts to show the continuity of the map $\varrho \mapsto \varrho_e[\varrho]$ in the appropriate topology. Using the convexity of the free energy, along with various formulations of the Klein inequality, we show that the map is of H\"older regularity 1/8  in the space $\calJ_2$ (this is not an optimal exponent), which finally leads to the existence of solutions. As mentioned earlier, the uniqueness is an open question.

The proof of Theorem \ref{thliou2} is based on the introduction of the relative entropies between $\varrho$ and $\varrho_e[\varrho]$ and vice versa. This leads to an estimate that shows that the free energy is nonincreasing and that the collision term in \fref{liouville2} converges to zero. The situation is then similar to the one encountered in the long-time behavior of the solutions to kinetic equations  \cite{DesvillettesBGK, DesvillettesVillani}: the free energy is dissipated until $\varrho(t)$ becomes a \textit{local} quantum Maxwellian, \textit{i.e.} a solution to the moment problem with a local density constraint; the fact that this local quantum Maxwellian is actually a global quantum Maxwellian is then ensured by the free Liouville equation, which enables us to show that any local quantum Maxwellian solution to the free Liouville equation is necessarily a global one.

The rest of the paper is devoted to the proofs of our main theorems.
\section{Proof of Theorem \ref{regA}} \label{proofregA}

The key ingredient is the representation formula obtained in the proposition below. It will allow us to obtain the $L^r$ regularity of the chemical potential $A$ as well as some refined estimates on the spectral elements of the operator $H+A$. 

\begin{proposition} \label{expA}
Consider a density $n\in \Hunper$ such that $n>0$ on $[0,1]$ and denote by $\Delta n$ the Laplacian of $n$ in the distribution sense. Then $A \in \H^{-1}_{per}$ admits the expression
\be \label{repform}
A=-\frac{1}{n}\left(-\frac{1}{2}\Delta n+ \sum_{p\in\NN^*} \rho_p |\nabla \phi_p|^2+\sum_{p\in\NN^*} \left(\rho_p \log \rho_p \right) |\phi_p|^2\right),
\ee
where  $(\rho_p, \phi_p)_{p \in \NN^*}$ is the spectral decomposition of the minimizer $\varrho[n]$.
\end{proposition}

Note that \fref{repform} leads to the equality
$$
\int_0^1 n(x) A(x) dx= - \Tr \big(\sqrt{H} \varrho[n] \sqrt{H}\big) - \Tr \big(\varrho[n] \log \varrho[n]\big)=-F(\varrho[n])- \int_0^1 n(x) dx, 
$$
which will be recovered in a different manner in Section \ref{further}.
\begin{proof}The proof consists in exhibiting $\Delta n$ in \fref{defAA}. For this, let us define $A_1$ and $A_2$ by, for all $\psi \in \Hunper$:
\bea
\left(A_1,\psi\right)_{\H^{-1}_{per},\Hunper}&:=&-\Tr\left(\frac{\psi}{n} (\varrho\log \varrho)\right)\nonumber ,\\
 \left(A_2,\psi\right)_{\H^{-1}_{per},\Hunper}&:=&-\sum_{p\in\NN^*}\left((\sqrt{H}\sqrt{\varrho})\phi_p,\sqrt{H}\left(\frac{\psi}{n}\sqrt{\varrho}\phi_p\right)\right)_{L^2}\label{defA2},
\eea
where $\varrho \equiv \varrho[n]$. The identification of $A_1$ is straightforward: since $\varrho \log \varrho \in \calJ_1$ according to \fref{estimsolmom1}, we can write
$$
\Tr\left(\frac{\psi}{n} (\varrho\log \varrho)\right)=\int_0^1 \frac{\psi}{n} n[\varrho\log \varrho] dx=\int_0^1 \frac{\psi}{n}\sum_{p\in\NN^*} (\rho_p \log \rho_p ) |\phi_p|^2 dx,
$$
where the sum converges in $L^1(0,1)$. This yields the last term in the expression of $A$. Regarding $A_2$, we first remark that it is real-valued since  $A$ and $A_1$ are real-valued. Choosing then  a real-valued test function $\psi$ and taking the real part of \fref{defA2} leads to
$$
 \left(A_2,\psi\right)_{\H^{-1}_{per},\Hunper}=-L\left(\psi\right):=-\Re \sum_{p\in\NN^*} \rho_p\left((\sqrt{H}\phi_p,\sqrt{H}\left(\frac{\psi}{n}\phi_p\right)\right)_{L^2}.
$$
In order to justify formal computations, we introduce the following regularized functional, for all $\psi \in \Hunper$:
$$
 L_{N,\eps}\left(\psi\right)=\Re \sum_{p=1}^N \rho_p\left((\sqrt{H}\phi^\eps_p,\sqrt{H}\left(\frac{\psi}{n}\phi^\eps_p\right)\right)_{L^2},
$$
where $\phi^\eps_p \in \calC^\infty_{per}([0,1])$ (the set of periodic $\calC^\infty$ functions with periodic derivatives) is such that $\phi^\eps_p \to \phi_p$ in $\H^1$ as $\eps \to 0$, for $p=1, \cdots, N$. It is not difficult to see that 
$$
\lim_{N \to \infty} \lim_{\eps \to 0 } L_{N,\eps}\left(\psi\right)=L\left(\psi\right), \qquad \forall \psi \in \Hunper.
$$
We then have
$$
 L_{N,\eps}\left(\psi\right)=\Re \sum_{p=1}^N \rho_p\left(H\phi^\eps_p,\left(\frac{\psi}{n}\phi^\eps_p\right)\right)_{L^2}=-\int_0^1 \frac{\psi}{n}  \, \Re \left(\sum_{p=1}^N \rho_p\overline{\Delta \phi^\eps_p} \phi_p^\eps  \right)dx.
$$
Define now
$$
n_{N,\eps}:=\sum_{p=1}^N \rho_p |\phi_p^\eps|^2,
$$
so that
 $$\Delta n_{N,\eps}= 2 \sum_{p=1}^N \rho_p |\nabla \phi^\eps_p|^2+ 2 \Re \left(\sum_{p=1}^N \rho_p \overline{\Delta \phi^\eps_p} \phi^\eps_p\right).$$
We can then recast $L_{N,\eps}$ as
$$
L_{N,\eps}\left(\psi\right)=\int_0^1  \frac{\psi}{n} \left(-\frac{1}{2}\Delta n_{N,\eps}+ \sum_{p=1}^N \rho_p |\nabla \phi^\eps_p|^2\right)dx.
$$
Owing to the inclusion $\Hunper \subset L^\infty(0,1)$, it is clear that $\nabla n_{N,\eps} \to \nabla n_{N} $  in $L^2(0,1)$ as $\eps \to 0$. Taking the limit $\eps \to 0$ after an integration by parts, we find
$$
\lim_{\eps \to 0} L_{N,\eps}\left(\psi\right)=\frac{1}{2} \int_0^1 \nabla \left(\frac{\psi}{n} \right) \cdot \nabla  n_{N}\,dx+ \int_0^1\frac{\psi}{n} \sum_{p=1}^N \rho_p |\nabla \phi_p|^2dx,
$$
with obvious notation for $n_N$. Since
$$\sum_{p \in \NN^*}\rho_p\left\| \phi_p \right\|_{L^2}^2+\sum_{p \in \NN^*}\rho_p\left\|\nabla \phi_p \right\|_{L^2}^2=\Tr \varrho+\Tr \big(\sqrt{H}\varrho\sqrt{H}\big)<+\infty,$$
the series 
$$
\nabla n_N=\sum_{p=1}^N \rho_p \nabla |\phi_p|^2 \qquad \textrm{and} \qquad \sum_{p=1}^N \rho_p |\nabla \phi_p|^2,
$$
are absolutely convergent in $L^1(0,1)$, and we can pass to the limit as $N\to +\infty$ to obtain
$$
 \left(A_2,\psi\right)_{\H^{-1}_{per},\Hunper}=-\frac{1}{2} \int_0^1 \nabla \left(\frac{\psi}{n} \right) \cdot \nabla  n \,dx-\int_0^1\frac{\psi}{n} \sum_{p=1}^\infty \rho_p |\nabla \phi_p|^2dx.
$$
Gathering the latter result and the expression of $A_1$ then ends the proof.
\end{proof}

 To prove Theorem \ref{regA}, we will need some estimates on the eigenvectors $\phi_p$ and eigenvalues $\mu_p$ of the quadratic form $Q_A$ defined in \fref{qu}, that we state in the following lemma. 
\begin{lemma}(Estimates for the eigenvalues and eigenvectors of the minimizer). Denote by $(\rho_p,\phi_p)_{p \in \NN^*}$ the eigenvalues and eigenvectors of the minimizer $\varrho[n]$. For $\rho_p=e^{-\mu_p}$, and $A \in \H^{-1}_{per}$ the chemical potential, we have the following statements, for all $p \in \NN^*$:
\begin{itemize}
\item If $A \in L^r$ with $r=1$ or $r=2$, then
\be \label{estimdeltaL1}
\|\Delta \phi_p\|_{L^r} \leq |\mu_p|+C \|A\|_{L^r} \left(1+\|\nabla \phi\|^{1/2}_{L^2}\right).
\ee
\item If $\lambda_p=(2 \pi p)^2$, then
\be \label{controlmu}
\frac{1}{2} \lambda_p-C \|A\|^4_{\H^{-1}_{per}}-C \leq \mu_p \leq \frac{3}{2} \lambda_p +C \|A\|^4_{\H^{-1}_{per}}+C.
\ee
\item Finally,
\be \label{estimgrad}
\| \nabla \phi_p\|^2_{L^2} \leq  C|\mu_p|+C\|A\|^4_{\H^{-1}_{per}} +C.
\ee
\end{itemize}
Above, $C$ is a constant independent of $p$ and $A$.
\end{lemma}
\begin{proof}
First of all, since $A$ is assumed to belong to $L^1$ for \fref{estimdeltaL1}, we have the identification
$$ (A, \psi)_{\H^{-1}_{per},\Hunper}=\int_0^1 A \psi dx $$
and we find, for all $\varphi \in \Hunper$,
\begin{equation}
\label{a1}
Q_A(\varphi,\phi_p)=(\nabla \varphi, \nabla \phi_p)+\int_0^1\overline{\varphi}A \phi_pdx=\mu_p(\varphi,\phi_p),
\end{equation}
which implies that 
\be \label{vpe}-\Delta \phi_p=(\mu_p-A)\phi_p\ee
 in the distribution sense. Hence, using the Gagliardo-Nirenberg inequality for $\H^1$ functions (with $L^2$ norm equal to one),
\be \label{gag}
 \|\phi\|_{L^\infty} \leq C \|\phi\|^{1/2}_{L^2}\|\phi\|^{1/2}_{\H^1}=C\left(1+\|\nabla \phi\|^{1/2}_{L^2}\right),
\ee
we deduce \eqref{estimdeltaL1}. In order to prove \eqref{controlmu}, we write, for any function $\varphi \in \Hunper$ with $L^2$ norm equal to one,  
\begin{eqnarray}
\left|\int_0^1 A |\varphi|^2dx \right|&\leq& \|A\|_{\H^{-1}_{per}}\| |\varphi|^2 \|_{\H^1}\leq C \|A\|_{\H^{-1}_{per}}\ \left(\| \varphi \|_{L^4}^2+\|\varphi\|_{L^\infty}\left\|\nabla \varphi\right\|_{L^2}\right)\nonumber\\
&\leq & C \|A\|_{\H^{-1}_{per}}\left(\left\| \varphi\right\|_{\H^1}^{1/2}+\left\| \varphi \right\|^{3/2}_{\H^1} \right)\nonumber\\
&\leq&\frac{1}{2}\left\|\nabla \varphi \right\|_{L^2}^2+C\|A\|^4_{\H^{-1}_{per}}+C.\label{a2}
\end{eqnarray}
This yields
$$
\frac{1}{2} (\nabla\varphi,\nabla \varphi)-C \|A\|^4_{\H^{-1}_{per}}-C \leq Q_{A}(\varphi,\varphi) \leq \frac{3}{2}(\nabla \varphi,\nabla \varphi)+C \|A\|^4_{\H^{-1}_{per}} +C
$$
and \eqref{controlmu} follows from the minimax principle. Finally, \fref{estimgrad} is deduced from \eqref{a1} and \eqref{a2} by choosing $\varphi=\phi_p$. 
\end{proof}

\begin{proof}[Proof of Theorem \ref{regA}]  Let $n\in \Hunper$ such that $n(x)\geq \underline n>0$ on $[0,1]$ and $\Delta n \in W^{k,r}(0,1)$. Clearly, we have $\frac{\Delta n}{n}\in W^{k,r}(0,1)$ so, by Proposition \ref{expA}, it suffices to show that the two series in the expression \eqref{repform} of $A$ are convergent in $W^{k,r}(0,1)$. We first notice that, by Proposition \ref{severalestimates}, both series are absolutely convergent in $L^1$. Hence,  we already have $A \in L^1(0,1)$ with
\bea \nonumber 
\|A\|_{L^1}&\leq& \frac{C}{\underline{n}} \left( \|\Delta n \|_{L^1} +\Tr\sqrt{H}\varrho \sqrt{H}+\Tr|\varrho \log \varrho |\right)\\ 
&\leq& \frac{C}{\underline{n}} \left(\|\Delta n \|_{L^2} +\mathfrak{H}_0(n)\right)<\infty.
\label{estimAL1}
\eea

Let us estimate the $L^r$ norm of $A$. We treat the case $r=2$ separately, since refined estimates will be needed for the analysis of the Liouville equation.
\bs

\paragraph{\textit{The case $k=0$, $r=2$.}} By the Gagliardo-Nirenberg inequality
$$
\|\nabla \phi\|_{L^4} \leq C \|\nabla \phi\|^{1/2}_{L^2}\|\nabla \phi\|^{1/2}_{W^{1,1}}
$$
and Cauchy-Schwarz, we have
\begin{align*}
\sum_{p=1}^N \rho_p \| \nabla\phi_p \|^2_{L^4}
&\leq C \left(\sum_{p=1}^N \rho_p  \| \nabla\phi_p \|^2_{L^2}\right)^{1/2}
\left(\sum_{p=1}^N \rho_p  \| \nabla\phi_p \|^2_{W^{1,1}}\right)^{1/2}\\
&\leq C\bigg(\Tr\sqrt{H}\varrho \sqrt{H}\bigg)^{1/2} \left(\sum_{p=1}^N \rho_p  (\| \nabla\phi_p \|^2_{L^1}+\| \Delta \phi_p \|^2_{L^1})\right)^{1/2}.
\end{align*}
Let
$$
S=\sum_{p=1}^N \rho_p  \Big(\| \nabla\phi_p \|^2_{L^1}+\| \Delta \phi_p \|^2_{L^1}\Big).
$$
Using  \fref{estimdeltaL1} and the Young inequality, recalling that $\rho_p=e^{-\mu_p}$, we find
\begin{align*}
S&  \leq C \sum_{p=1}^N \rho_p  |\mu_p|^2+C \sum_{p=1}^N \rho_p \Big(1+\|A\|_{L^1}^2\Big)\Big(1 +\| \nabla \phi_p \|_{L^2}^2\Big)\\
& \leq C \sum_{p=1}^N e^{-\mu_p}  |\mu_p|^2+C\Big(1+\|A\|_{L^1}^2\Big)\Big(\Tr \varrho+\Tr \sqrt{H}\varrho \sqrt{H}\Big).
\end{align*}
The eigenvalues $e^{- \mu_p}$ can then be estimated thanks to \fref{controlmu},
$$e^{- \mu_p} \leq C e^{- \lambda_p/2}  \exp \big(C \|A\|^4_{\H^{-1}_{per}}\big) \qquad \textrm{and} \qquad |\mu_p|^2 \leq C\big(1+\lambda_p^2+\|A\|^8_{\H^{-1}_{per}}\big). $$ Therefore, using that $\lambda_p=(2 \pi p)^2$, 
\begin{align}
\sum_{p=1}^N e^{-\mu_p}  |\mu_p|^2&\leq C \exp\big(C \|A\|^4_{\H^{-1}_{per}}\big) \sum_{p=1}^N e^{-\lambda_p/2} \left( 1+\lambda_p^2+\|A\|^8_{\H^{-1}_{per}}\right)\nonumber\\
&\leq C \exp\big(C \|A\|^4_{\H^{-1}_{per}}\big).\label{a3}
\end{align}
Hence, going back to the definition of $S$, we find
$$
S\leq C \exp\big(C \|A\|^4_{\H^{-1}_{per}}\big)+C\big(1+\|A\|^2_{L^1}\big)\| \varrho \|_\calE,
$$
which yields
$$
 \sum_{p=1}^N \rho_p \| \nabla\phi_p \|^2_{L^4} \leq C\| \varrho \|^{1/2}_{\calE} \exp\big(C\|A\|^4_{\H^{-1}_{per}}\big)+C \left(1+\|A\|_{L^1} \right) \| \varrho \|_\calE.
$$
This implies that the series $\sum_p \rho_p | \nabla\phi_p |^2$ is absolutely convergent in $L^2(0,1)$.  Using \fref{estimsolmom1}, \fref{estimsolmom2} and \fref{estimAL1} together with the Young inequality, we obtain finally
\begin{equation}
\sum_{p=1}^\infty \rho_p \| \nabla\phi_p \|^2_{L^4}\leq C \mathfrak{H}_0(n)\left(1+\frac{1}{\underline{n}} \left({\| \Delta n \|_{L^2}}+\mathfrak{H}_0(n)\right)\right)+C \exp\left( C \mathfrak{H}_1(n)^4\right). \label{estimserie1} 
\end{equation}
We estimate now the entropy term in \eqref{repform}. Using a Gagliardo-Nirenberg inequality and the fact that $\rho_p=e^{-\mu_p}$, we get
\bea \nonumber
\sum_{p=1}^N | \rho_p \log \rho_p| \| \phi_p \|^2_{L^4} &\leq& C \sum_{p=1}^N |\rho_p \log \rho_p| \| \phi_p \|_{\H^1} \\\nonumber
&\leq& C \Bigg(\Tr \varrho+\Tr \sqrt{H}\varrho \sqrt{H}\Bigg)^{1/2} \Bigg(
\sum_{p=1}^N e^{- \mu_p} |\mu_p|^2  \Bigg)^{1/2}\\
&\leq& C \|\varrho\|_{\calE}^{1/2} \exp(C \|A\|^4_{\H^{-1}_{per}})\nonumber\\
&\leq& C \sqrt{\mathfrak{H}_0(n)} \exp\left( C \mathfrak{H}_1(n)^4\right)
, \label{estimserie2}
\eea
where we used \eqref{a3}, \eqref{estimsolmom1} and \eqref{estimsolmom2}. The series $ \sum_p \rho_p \log \rho_p |\phi_p|^2$ is absolutely convergent in $L^2(0,1)$ and, gathering \fref{estimserie1}, \fref{estimserie2} and \eqref{repform}, we obtain the estimate of Theorem \ref{regA}. This concludes the case $k=0$, $p=2$.\\

\paragraph{\textit{Case $k=0$, $p\in [1,\infty]$}} Let us prove that the two series in \eqref{repform} are converging in $L^\infty(0,1)$. Recall first the Sobolev embedding $W^{1,1} \subset L^\infty$, which, together with \fref{estimdeltaL1} yields
\begin{align*}
\|\nabla \phi_p\|_{L^\infty} \leq C \|\nabla \phi_p\|_{W^{1,1}} &\leq C (\|\Delta \phi_p\|_{L^1}+\|\nabla \phi_p\|_{L^1})\\
&\leq  C (|\mu_p|+\|A\|_{L^1}^2+\|\nabla \phi_p\|_{L^2}).
\end{align*}
Therefore,
\bee
\sum_{p=1}^N \rho_p \| \nabla\phi_p \|^2_{L^\infty} &\leq& C \sum_{p=1}^N \rho_p (|\mu_p|^2+\|A\|_{L^1}^4+\|\nabla \phi_p\|_{L^2}^2)\\
 &\leq & C \sum_{p=1}^N e^{- \mu_p} |\mu_p|^2+ C\|A\|_{L^1}^4\Tr \varrho +C \Tr\sqrt{H}\varrho \sqrt{H},
\eee
which shows that the series $\sum_p \rho_p | \nabla\phi_p |^2$ is absolutely convergent in $L^\infty(0,1)$, since $\mu_p$ satisfies \fref{controlmu}. Similarly, using again \fref{gag} and Cauchy-Schwarz, we have
\bee
\sum_{p=1}^N |\rho_p \log \rho_p| \| \phi_p \|^2_{L^\infty} &\leq& C \sum_{p=1}^N |\rho_p \log \rho_p| \left(1+\|\nabla \phi_p \|_{L^2}\right)\\
&\leq& C \Bigg(\sum_{p=1}^N e^{- \mu_p} |\mu_p|^2 \Bigg)^{1/2}\Bigg(\Tr \varrho+\Tr\sqrt{H}\varrho \sqrt{H}\Bigg)^{1/2}
\eee
so the series $ \sum_p \rho_p \log \rho_p |\phi_p|^2$ is absolutely convergent in $L^\infty(0,1)$. We have  proved that if $\Delta n \in L^1(0,1)$, then both series
$$
\sum_{p\in\NN^*} \rho_p |\nabla \phi_p|^2\quad\mbox{and}\quad\sum_{p\in\NN^*} \rho_p \log \rho_p |\phi_p|^2
$$
are convergent in $L^\infty(0,1)$.
Hence, we conclude with Proposition \ref{expA} that $\Delta n \in L^p(0,1)$ implies $A \in L^p(0,1)$ for $p \in [1,\infty]$. \\

\paragraph{\textit{General case and conclusion}} Let us prove by induction that, if $\Delta n \in W^{k,r}$, then $A \in W^{k,r}$. The case $k=0$ was just addressed. Suppose then that $\Delta n \in W^{k,r}$ for $k \geq 1$, and that $A \in W^{k-1,r}$. In order to obtain that $A \in W^{k,r}$, we differentiate \eqref{repform} $k$ times. We shall only prove that the following series is absolutely convergent in $L^\infty$, the other terms can be estimated similarly and are left to the reader: 
$$
T_N:= \sum^N_{p=1} \rho_p (\partial^{k+1}_x \phi_p )(\overline{\partial_x \phi_p}). 
$$ 
Using \eqref{vpe}, we get the estimate
$$
\| \partial^2_x \phi_p\|_{L^1} \leq \left(|\mu_p| +\|A\|_{L^\infty}\right)\|\phi_p\|_{L^1}
$$
and, using that $A \in  W^{k-1,r}\cap L^\infty$ and the interpolation inequality 
$$\|\pa_x \phi_p\|_{L^1}\leq C\|\phi_p\|_{L^1}^{1/2}\|\pa_x^2\phi_p\|_{L^1}^{1/2},$$
we obtain
$$
\| \phi_p \|_{W^{2,1}}\leq C\left(1+|\mu_p|\right)\|\phi_p\|_{L^1} \leq C\left(1+|\mu_p|\right).
$$
Now, coming back to \eqref{vpe}, we get, for all $1\leq \ell \leq k-1$, 
\begin{align*}
\| \partial^{\ell+2}_x \phi_p\|_{L^1} &\leq \left(|\mu_p| +\|A\|_{L^\infty}\right)\| \partial^\ell_x \phi_p\|_{L^1}+C \|A\|_{W^{\ell,1}} \| \phi_p \|_{W^{\ell-1,\infty}}\\
&\leq |\mu_p| \| \partial^\ell_x \phi_p\|_{L^1}+C \|A\|_{W^{\ell,1}} \| \phi_p \|_{W^{\ell,1}},
\end{align*}
where we used that $W^{\ell,1}\hookrightarrow W^{\ell-1,\infty}$. This leads to
$$
\| \phi_p \|_{W^{\ell+2,1}}\leq  C(1+|\mu_p|) \| \phi_p \|_{W^{\ell,1}}.
$$
After iterating, this yields
\begin{align*}
\|\phi_p\|_{W^{\ell+2,1}} &\leq C(1+|\mu_p|)^{\ell/2+1} &&\mbox{if $\ell$ is even,}\\
\|\phi_p\|_{W^{\ell+2,1}} &\leq C(1+|\mu_p|)^{(\ell-1)/2+1} \| \phi_p \|_{W^{1,1}} &&\mbox{if $\ell$ is odd,}\\
&\leq C(1+|\mu_p|)^{(\ell-1)/2+1} \| \phi_p \|_{L^2}^{1/2} \| \phi_p \|_{W^{2,1}}^{1/2}&&\mbox{(by interpolation)}\\
&\leq C(1+|\mu_p|)^{\ell/2+1} 
\end{align*}
 Therefore, using that $W^{k+1,1}\hookrightarrow W^{1,\infty}$, we get
\bee
\| T_N\|_{L^1} \leq C \sum_{p \in \NN^*} \rho_p \|\phi_p\|_{W^{k+1,1}}\|\na \phi_p\|_{L^\infty} &\leq &C \sum_{p \in \NN^*} \rho_p \|\phi_p\|_{W^{k+1,1}}^2 \\
&\leq &C \sum_{p \in \NN^*} \rho_p  (1+|\mu_p|)^{k+1}<+\infty\\
\eee
since $\rho_p=e^{-\mu_p}$ and $\mu_p$ satisfies \fref{controlmu}: we obtain that $A \in W^{k,1}$. Hence, the above estimates are also valid for $\ell=k$ and we conclude with
\begin{align*}
\| T_N\|_{L^\infty} \leq C \sum_{p \in \NN^*} \rho_p \|\phi_p\|_{W^{k+1,\infty}}^2&\leq C \sum_{p \in \NN^*} \rho_p \|\phi_p\|_{W^{k+2,1}}^2 \\
&\leq C\sum_{p \in \NN^*} \rho_p (1+|\mu_p|)^{k+2}<+\infty.
\end{align*}

This ends the proof of Theorem \ref{regA}.
\end{proof}
\section{More results on the moment problem} \label{further}

An application of Theorem \ref{regA} is the following proposition, which provides additional regularity on the minimizer $\varrho[n]$ if $A$ belongs to $L^2$.
\begin{proposition} \label{regmin} (Regularity of the minimizer) Consider a density $n\in \Hunper$ such that $n\geq \underline{n}>0$ on $[0,1]$ and $\Delta n \in L^2$. Then, for any $\alpha>0$, the minimizer $\varrho[n]=e^{-H_A}$, with $H_A=H+A$, belongs to $\calH$, satisfies $(\varrho[n])^\alpha \in \calJ_1$ and satisfies the estimate, for any $a \in \RR$,
\be \label{estimFA}
\Tr\big((H+a\II) (\varrho[n])^\alpha (H+a\II) \big) \leq C \exp\big( C \|A\|^4_{\H^{-1}_{per}}\big) \big(1+\|A\|_{L^2}^2\big),
\ee
where $C$ is a constant independent of $n$. Moreover, $H \varrho[n] \in \calJ_1$ and $\varrho[n] H\in \calJ_1$.
\end{proposition}
\begin{proof} First of all, by Theorem \ref{regA}, we have $A \in L^2$. Then, direct computations yield
\bea
\Tr \big((H+a\II) (\varrho[n])^\alpha (H+a\II) \big)&=&\sum_{p\in \NN^*} (\rho_p)^\alpha \| (H+a\II) \phi_p \|^2_{L^2}\nonumber\\
&\leq & 2\sum_{p\in \NN^*} (\rho_p)^\alpha \| \Delta \phi_p \|^2_{L^2}+2a^2\sum_{p\in \NN^*}(\rho_p)^\alpha.\label{a4}
\eea
According to \fref{estimdeltaL1}, \fref{controlmu}, and \fref{estimgrad}, we can control the first term of the right-hand-side by
\begin{align*}
&\sum_{p\in \NN^*} (\rho_p)^\alpha \| \Delta \phi_p \|^2_{L^2}\\
&\hspace{1cm}\leq C \sum_{p\in \NN^*} e^{-\alpha \mu_p} \Big( |\mu_p|^2+\|A\|^2_{L^2}  \left(1+\|\nabla \phi_p\|_{L^2}\right)\Big)\\
&\hspace{1cm}\leq C\exp\Big(C \|A\|^4_{\H^{-1}_{per}}\Big)\sum_{p\in \NN^*} e^{-C \alpha p^2} \Big( |\mu_p|^2+\|A\|^2_{L^2} \big(1+|\mu_p|^{1/2}+\|A\|^2_{\H^{-1}_{per}}\big) \Big)\\
&\hspace{1cm}\leq C\exp\Big(C \|A\|^4_{\H^{-1}_{per}}\Big) \Big(1+\|A\|^2_{L^2}\Big).
\end{align*}
With  \fref{controlmu}, the second term of the r.h.s. of \eqref{a4} is straightforwardly controlled by
$$
 \sum_{p\in \NN^*}(\rho_p)^\alpha \leq C\exp\Big(C \|A\|^4_{\H^{-1}_{per}}\Big) \sum_{p\in \NN^*} e^{-C \alpha p^2} \leq  C\exp\Big(C \|A\|^4_{\H^{-1}_{per}}\Big).
$$
The fact that $H \varrho[n] \in \calJ_1$ is easily established by using the polar decomposition $H \varrho[n]= U|H \varrho[n]|$, where $U$ is a partial isometry on $L^2$ {\em i.e.} an isometry on $\mbox{Ker\,}U^\perp=\overline{\mbox{Ran\,}|H\varrho[n]|}$),  and the following computation
\bee
\Tr |H \varrho[n]|&=&\Tr \big(U^* H \varrho[n] \big)=\Tr \big(U^* H (H+\II)^{-1}(H+\II) \varrho[n](H+\II)(H+\II)^{-1} \big)\\
&\leq& \|U^* H (H+\II)^{-1}\|_{\calL(L^2)} \Tr \big((H+\II) \varrho[n](H+\II)\big)\|(H+\II)^{-1}\|_{\calL(L^2)}\\
&\leq& C\Tr \big((H+\II) \varrho[n](H+\II)\big)<\infty.
\eee
Therefore, $H \varrho[n] \in \calJ_1$ and its adjoint $\varrho[n] H$ is also trace-class. This ends the proof.
\end{proof}

In the analysis of the Liouville equation, we will need the following result.
\begin{lemma} \label{H2n}
Let $\varrho \in \calH$ with $\varrho \geq 0$. Then, the corresponding local density $n[\varrho]$ belongs to $\H^2(0,1)$, and we have the estimate
$$
\|\Delta n[\varrho]\|_{L^2} \leq C \| \varrho \|^{1/2}_{\calE} \big(\Tr H\varrho H \big)^{1/2}.
$$
\end{lemma}
\begin{proof} Using the embedding $\H^1(0,1) \subset L^\infty(0,1)$, it is direct to deduce  from \eqref{a5}  that  $n:=n[\varrho] \in \H^1$. Let us show that $\Delta n\in L^2$. We have
$$\Delta n= 2 \sum_{p=1}^\infty \rho_p |\nabla \phi_p|^2+ 2 \Re \left(\sum_{p=1}^\infty \rho_p \overline{\Delta \phi_p} \phi_p\right),$$
where $(\rho_p, \phi_p)_{i \in \NN^*}$ denote the spectral elements of $\varrho$. Let us show that these series are  absolutely  convergent in $L^2$. First of all, by the Gagliardo-Nirenberg inequality
$$\|\nabla \phi_p\|_{L^4}\leq C\|\nabla \phi_p\|_{L^2}^{1/2}\|\Delta\phi_p\|_{L^2}^{1/2},$$
we find
\bee
\sum_{p=1}^\infty \rho_p \|\nabla \phi_p\|^2_{L^4} &\leq & C \sum_{p=1}^\infty \rho_p \|\nabla \phi_p\|_{L^2}\|\Delta \phi_p\|_{L^2}\\
&\leq & C\|\varrho\|_{\calE}^{1/2}\left(\Tr H\varrho H\right)^{1/2}.
\eee
Similarly,
\bee
\sum_{p=1}^\infty \rho_p \|\overline{\phi_p} \Delta \phi_p \|_{L^2} &\leq & \left(\sum_{p=1}^N \rho_p \| \phi_p\|^2_{L^\infty}\right)^{1/2} \left(\sum_{p=1}^N \rho_p \| \Delta \phi_p\|^2_{L^2}\right)^{1/2},\\
&\leq & C\|\varrho\|_{\calE}^{1/2}\left(\Tr H\varrho H\right)^{1/2},
\eee
where we used the embedding $\H^1(0,1) \subset L^\infty(0,1)$. This ends the proof of the lemma.
\end{proof}

The next lemma provides us with Lieb-Thirring type inequalities that will be important in the analysis of quantum Maxwellians. Note that the lemma applies to self-adjoint operators that do not have a particular sign.

\begin{lemma} \label{LT} Let $\varrho \in \calE$ with $\varrho=\varrho^*$. Then, the following estimates hold:
\begin{align}
\label{ninfty}\|n[\varrho]\|_{L^\infty} \leq C \| \varrho\|^{1/4}_{\calJ_2}  \|\varrho\|_{\calE}^{3/4}\\
\label{gradnl2}\|\nabla n[\varrho]\|_{L^2} \leq C \| \varrho\|^{1/4}_{\calJ_1}  \|\varrho\|_{\calE}^{3/4}.
\end{align} 
\end{lemma}
\begin{proof} Denote by $(\rho_k,\phi_k)_{k\in \NN^*}$ the spectral elements of $\varrho$. For the first estimate, we have, using \fref{gag},
\bee
\|n[\varrho]\|_{L^\infty} &\leq & \sum_{k \in \NN^*} |\rho_k| \|\phi_k\|^2_{L^\infty} \leq C \sum_{k \in \NN^*} |\rho_k| \|\phi_k\|_{\H^1}\\
&\leq & C\left(\sum_{k \in \NN^*} |\rho_k|\right)^{1/2} \left(\sum_{k \in \NN^*} |\rho_k|\|\phi_k\|^2_{\H^1}\right)^{1/2}.\\
\eee
The second term in the r.h.s. is controlled by $\|\varrho\|^{1/2}_{\calE}$, while we write for the first one
\bee
\sum_{k \in \NN^*} |\rho_k| &\leq& \left(\sum_{k \in \NN^*} |\rho_k|^{2/3} \right)^{3/4} \left(\sum_{k \in \NN^*} |\rho_k|^2 \right)^{1/4}\\
&\leq& \left( \Tr \sqrt{H} |\varrho| \sqrt{H}\right)^{1/2} \|\varrho\|_{\calJ_2}^{1/2},
\eee
where we used Lemma \ref{lieb} in the second line. Regarding the second estimate, we have directly, using once more \fref{gag},
\bee
\|\nabla n[\varrho]\|_{L^2} &\leq & 2 \sum_{k \in \NN^*} |\rho_k| \|\phi_k\|_{L^\infty}\|\nabla \phi_k\|_{L^2} \leq C \sum_{k \in \NN^*} |\rho_k| \| \phi_k\|^{3/2}_{\H^1} \\
&\leq & C \left(\sum_{k \in \NN^*} |\rho_k|\right)^{1/4} \left(\sum_{k \in \NN^*} |\rho_k|\|\phi_k\|^2_{\H^1}\right)^{3/4}.\\
\eee
\end{proof}

The next proposition shows the continuity of the application $n \to \varrho[n]$ and is a crucial ingredient in the proof of Theorem \ref{thliou}. It is mostly a consequence of the convexity of the free energy and relies on the characterization of the minimizer as a quantum Maxwellian.

\begin{proposition} \label{contsol} Let $n_1$ and $n_2$ two densities in $\Hunper$ such that, for $i \in \{1,2\}$, $\Delta n_i$ belongs to $L^2$, and $n_i(x) > 0$, $\forall x \in [0,1]$. Let $\varrho[n_i]=e^{-(H+A_i)}$ be the solution to the moment problem for the density $n_i$. Then, the following inequality holds: 
$$
 \|\varrho[n_1]-\varrho[n_2]\|_{\calJ_2}^2 \leq C (A_2-A_1,n_1-n_2),
$$
where the constant $C$ is independent of $n_1$ and $n_2$.
\end{proposition}
\begin{proof} The proof hinges  upon two versions of the Klein inequality. The first result applies to unbounded self-adjoint operators.
\begin{theorem}[\cite{ruelle}, Chap. 2, \S2.5.2 (Klein inequality I)]\label{klein2} Let $\varphi$ be a convex function. Then, for two (unbounded) self-adjoint operators $H_1$ and $H_2$ with spectrum in the domain of definition of $\varphi$, the following inequality holds:
$$
\Tr\left(\varphi(H_1)-\varphi(H_2)-(H_1-H_2)\varphi'(H_2)\right) \geq 0.
$$
\end{theorem}
The second result refines the first result when the operators are bounded. 
\begin{theorem} [\cite{LS}, Theorem 3 (Klein inequality II)] \label{klein} Let $\varphi \in \calC^0([0,M],\RR)$ such that $\varphi'$ is not constant and is operator monotone on $(0,M)$. Then, for two nonnegative bounded self-adjoint operators $\varrho_1$ and $\varrho_2$ with spectrum in $[0,M]$, the following inequality holds:
$$
\Tr\left(\varphi(\varrho_1)-\varphi(\varrho_2)-\varphi'(\varrho_2)(\varrho_1-\varrho_2)\right) \geq C\Tr\left((1+|\varphi'(\varrho_2)|)(\varrho_1-\varrho_2)^2\right).
$$
Above, the constant $C>0$ only depends on $\varphi$.
\end{theorem}
Note that Theorem \ref{klein} is stated for $M=1$ in \cite{LS} but holds for any $M>0$. Let us apply this theorem to the entropy function $\varphi(x)=\beta(x)= x \log x-x$.
To this aim, let us first verify that the term
$$
T:=\Tr\big(\varphi'(\varrho_2)(\varrho_1-\varrho_2)\big)
$$
is finite. We have $\varphi'(\varrho_2)=-H-A_2$ and recall that, by Proposition \ref{regmin}, we have $H\varrho_1\in \calJ_1$ and $H\varrho_2\in \calJ_1$. Let us check that $A_j \varrho_i\in \calJ_1$ for $i,j\in\{1,2\}$. Using the polar decomposition $A_j \varrho_i= U|A_j \varrho_i|$, where $U$ is a partial isometry on $L^2$, we compute
\bee
\Tr |A_j \varrho_i|&=&\Tr \big(U^* A_j \varrho_i \big)=\Tr \big(U^* A_j (H+\II)^{-1}(H+\II) \varrho_i(H+\II)(H+\II)^{-1} \big)\\
&\leq& \|U^* A_j (H+\II)^{-1}\|_{\calL(L^2)} \Tr \big((H+\II) \varrho_i(H+\II)\big)\|(H+\II)^{-1}\|_{\calL(L^2)}\\
&\leq& C\|A_j\|_{L^2}\Tr \big((H+\II) \varrho_i(H+\II)\big)<\infty,
\eee
where we used \eqref{estimFA} with $\alpha=a=1$, Theorem \ref{regA}, and the fact that $(H+\II)^{-1}$ is a bounded operator from $L^2(0,1)$ to $L^\infty(0,1)$. Therefore, we have $A_j \varrho_i\in \calJ_1$. Moreover, if we introduce a regularization $A^k_j\in L^\infty(0,1)$ such that $A^k_j\to A_j$ as $k\to +\infty$, the same calculation yields
$$\left|\Tr \left((A_j^k-A_j)\varrho_i\right)\right|\leq C\|A_j^k-A_j\|_{L^2}\Tr \big((H+\II) \varrho_i(H+\II)\big)\to 0\mbox{ as }k\to +\infty,$$
thus
$$\Tr \left(A_j\varrho_i\right)=\lim_{k\to +\infty}\Tr \left(A_j^k\varrho_i\right)=\lim_{k\to +\infty}\left(A_j^k,n_i\right)=\left(A_j,n_i\right).$$
We have proved in particular that 
$$
T=-\Tr\big((H+A_2)(\varrho_1-\varrho_2)\big)=-\Tr\big(H (\varrho_1-\varrho_2)\big)-(A_2,n_1-n_2),
$$
where all the terms are finite. Noticing that
$$
\Tr\big(H \varrho_i\big)=\Tr\big(\varrho_i H\big)=\Tr\big(\sqrt{H} \varrho_i \sqrt{H}\big)
$$
and introducing the free energies $F(\varrho_i)$, we deduce from Theorem \ref{klein} with $\varphi(x)=\beta(x)= x \log x-x$ that
\begin{align}
\|\varrho[n_1]-\varrho[n_2]\|_{\calJ_2}^2&\leq \Tr\left((1+|\varphi'(\varrho_2)|)(\varrho_1-\varrho_2)^2\right)\nonumber\\
& \leq \Tr\big(\beta(\varrho_1)-\beta(\varrho_2)-\log(\varrho_2)(\varrho_1-\varrho_2)\big) \nonumber\\
&\quad =F(\varrho_1)-F(\varrho_2)+(A_2,n_1-n_2).\label{a7}
\end{align}
We can then recast $F(\varrho_i)$ in terms of $A_i$ and $n_i$ as follows:
\bee
F(\varrho_i)&=& \Tr \big( \varrho_i \log \varrho_i -\varrho_i\big) + \Tr \big( \varrho_i H\big)\\
&=&-\Tr \big( \varrho_i H\big) -(A_i+1,n_i)+ \Tr \big( \varrho_i H\big)\\
&=&-(A_i+1,n_i).
\eee
As a consequence,
\begin{align*}
&&F(\varrho_1)-F(\varrho_2)+(A_2,n_1-n_2)=(A_2+1,n_2)-(A_1+1,n_1)+(A_2,n_1-n_2)\\
&&=(A_2-A_1,n_1-n_2)+(A_2-A_1,n_2)+(1,n_2-n_1).
\end{align*}
We apply now Theorem \ref{klein2} to the operators $H_i:=H+A_i$, $i \in \{1,2\}$, with $\varphi(x)=e^{-x}$, and $\varphi(H_i)=\varrho[n_i]$. Note that, according to the Kato-Rellich theorem, $H_i$ is self-adjoint with domain $D(H)$ since $A_i \in L^2$ (we use here Theorem \ref{regA}). It follows that
$$
\Tr\big(\varrho[n_1]-\varrho[n_2]+(A_1-A_2)\varrho[n_2]\big) \geq 0,
$$
which is equivalent to
$$
(1,n_1-n_2) +(A_1-A_2,n_2)\geq 0.
$$
Therefore
$$
F(\varrho_1)-F(\varrho_2)+(A_2,n_1-n_2) \leq (A_2-A_1,n_1-n_2),
$$
and \eqref{a7} enables to conclude the proof.

\end{proof}

\begin{remark} \label{rem} Proposition \ref{contsol} is stated in terms of the solutions to the moment problem with constraints $n_1$ and $n_2$. It could have been directly stated in terms of quantum Maxwellians $\varrho_i=\exp(-(H+A_i))$, with only assumptions $A_i \in L^2$, $i \in \{1,2\}$, and with $n_i$ equal to the local trace of $\varrho_i$. In that context, it follows easily from \fref{estimdeltaL1}-\fref{controlmu}-\fref{estimgrad} that $\varrho_i \in \calH$, and from \fref{ninfty}-\fref{gradnl2} that $n_i \in \H^1$, which shows that the term $(A_2-A_1,n_1-n_2)$ is well-defined.
\end{remark} 

As an application of the previous proposition, we estimate the difference $n_1-n_2$ when $n_i$ is the local density associated to an operator $\varrho_i \in \calH$.

\begin{corollary} \label{corcont} Let $\varrho_1$ and $\varrho_2$ two density operators in $\calH$. Let $M\in(0,\infty)$ be such that 
$$
\|\varrho_1 \|_{\calH}+\|\varrho_2\|_{\calH} \leq M, \quad \textrm{and} \quad  M^{-1}\leq n_i(x), \qquad \forall x \in [0,1], \qquad i=1,2,
$$
where $n_i$ is the local density associated to $\varrho_i$. Denote by $\varrho[n_i]$ the solution to the moment problem with density constraint $n_i$. Then, there exists a constant $C_{M}$, that only depends on $M$, such that
$$ \|\varrho[n_1]-\varrho[n_2]\|_{\calJ_2} \leq C_M \|\varrho_1-\varrho_2\|^{1/8}_{\calJ_2}.
$$
\end{corollary}
\begin{proof}
 First of all, since $\varrho_i \in \calH$, we deduce from Lemma \ref{H2n} that $\Delta n_i \in L^2$. Theorem \ref{regA} then shows that $A_i \in L^2$ with $\|A_i\|_{L^2} \leq C_M$. From Proposition \ref{contsol}, we have
$$
\|\varrho[n_1]-\varrho[n_2]\|_{\calJ_2}^2 \leq 2 C_M \|n_1-n_2\|_{L^2}. 
$$
Estimate \fref{ninfty} then yields
$$
\|n_1-n_2\|_{L^2} \leq \|n_1-n_2\|_{L^\infty} \leq C \| \varrho_1-\varrho_2\|_{\calJ_2}^{1/4}\| \varrho_1-\varrho_2\|_{\calE}^{3/4}.
$$
Since $\varrho_i \in \calH \subset \calE$, this yields the desired result. 
\end{proof}

\section{Proof of Theorem \ref{thliou}} \label{proofthliou}
In this section, we prove the existence of solution of the Liouville-BGK equation \eqref{liouville2}. The idea is to define a sequence $(\varrho_k)_{k \in \NN}$ of density operators such that $\varrho_0:=\varrho^0$ (the initial condition), and, for $k \geq 0$,
\be \label{mild}
 \partial_t \varrho_{k+1}=L_0(\varrho_{k+1})-\frac{1}{\tau}(\varrho_{k+1}-\varrho_e[\varrho_k]) \qquad \textrm{on } [0,T].
\ee
Above, $\varrho_e[\varrho_{k}]$ denotes the solution to the moment problem with constraint $n[\varrho_e[\varrho_{k}]]=n[\varrho_k]$, and $L_0(\varrho)=-i (\overline{H \varrho-\varrho H})$ is the infinitesimal generator of the group $\calL(t)\varrho=e^{-iHt} \varrho e^{iHt}$ defined in Section \ref{main}. The proof then goes as follows: the first step is to show that the sequence $(\varrho_k)_{k \in \NN}$ is well-defined; this amounts to solve the moment problem for each $k$ in order to obtain $\varrho_e[\varrho_k]$. For this, we need a bound from below on the density $n[\varrho_k]$. Once $\varrho_e[\varrho_k]$ is constructed, $\varrho_{k+1}$ is obtained from the following simple result given without proof (see \cite{arnold}, it is just a consequence of the fact that $\calL$ is an isometry on $\calJ_1$, $\calE$ or $\calH$):
\begin{lemma} \label{anton}
 Let $\varrho^0\in \calH$ and $\hat \varrho \in L^1((0,T),\calH)$. Then, the following density operator 
$$
\varrho(t)=e^{-\frac{t}{\tau}}\calL(t) \varrho^0+\int_0^t e^{-\frac{t-s}{\tau}}\calL(t-s) \hat \varrho(s) ds
$$
is the unique classical solution $\varrho \in \calC^0([0,T],\calH)$ of
$$ \partial_t \varrho_{k+1}=L_0(\varrho_{k+1})-\frac{1}{\tau}(\varrho_{k+1}-\hat \varrho) \qquad \textrm{on } [0,T].$$
Moreover, if $E$ is one of the following spaces, $\calJ_1$, $\calE$ or $\calH$, then this solution satisfies the estimate
\begin{equation}
\|\varrho(t)\|_{E} \leq e^{-\frac{t}{\tau}} \| \varrho^0\|_{E} +\int_0^t e^{-\frac{t-s}{\tau}} \|\hat \varrho(s)\|_{E}\, ds.
\label{estianton}
\end{equation}
\end{lemma}
The second step is to derive a uniform bound for $\varrho_k$ in $\calH$ that will provide some compactness. Using \fref{estimsolmom1} and the estimate of the lemma above, we show first that $\varrho_k(t)$ is uniformly bounded in $\calE$. This allows us to control the terms $\mathfrak{H}_0$ and $\mathfrak{H}_1$ in Proposition \ref{severalestimates} and, as a consequence, to control $A_k(t)$ (the chemical potential such that $\varrho_e[\varrho_k(t)]=\exp -(H+A_k(t))$)  uniformly in $\H_{per}^{-1}$. Combining then \fref{estimFA}, Theorem \ref{regA}, and Lemma \ref{H2n}, we obtain a \textit{sublinear} estimate of the form
$$
\|\varrho_e[\varrho_{k}(s)]\|_{\calH} \leq C \|\varrho_{k}(s)]\|_{\calH}.
$$
The last step is to pass to the limit in \fref{mild} using compactness arguments. The key ingredient is the continuity the map $\varrho \mapsto \varrho_e[\varrho]$ in the space $\calJ_2$ when $\varrho \in \calH$.

\subsection*{Step 1: Construction of the sequence $(\varrho_k)_{k \in \NN}$ and uniform bounds} In this first step, we prove the following proposition.

\begin{proposition} \label{unifbound}There exists a unique sequence $(\varrho_k)_{k \in \NN} \in \calC^0([0,T],\calH)$ defined by \fref{mild},  that satisfies the following uniform bound, for some constant $M$ independent of $k$,
\be \label{estimkH}
\sup_{t\in[0,T]}\|\varrho_k(t) \|_{\calH} \leq M, \qquad \forall k \in \NN,
\ee
and the following lower bound on the local density
$$
n[\varrho_k(t)](x) \geq \underline{n}>0, \quad \forall k \in \NN, \quad \forall x \in [0,1], \quad \forall t\in[0,T],
$$
where the constant $\underline{n}$ is independent of $k$. Finally, the solution $\varrho_e[\varrho_k(t)]$ to the moment problem with constraint $n[\varrho_k(t)]$ belongs to $\calH$ for every  $t \in [0,T]$ and verifies the estimate
\be \label{estimrhoe}
 \forall k \in \NN, \qquad \sup_{t\in[0,T]}\|\varrho_e[\varrho_k(t)] \|_{\calH} \leq M'',
\ee
for some $M''>0$ independent of $k$.
\end{proposition}
\begin{proof} 
We proceed by induction and first decompose $\varrho_{k+1}$ in \fref{mild} into
\be \label{decomprhok}
\varrho_{k+1}=\varrho^1+\varrho_{k+1}^2,
\ee
with 
\begin{equation}
\varrho^1(t)= e^{-\frac{t}{\tau}}\calL(t) \varrho^0, \qquad \varrho_{k+1}^2(t)=\frac{1}{\tau}\int_0^t e^{-\frac{t-s}{\tau}}\calL(t-s) \varrho_e[\varrho_{k}(s)] ds.
\label{b3}
\end{equation}
\bs

\paragraph{\it Initial step $k=0$} Owing to Assumption \ref{ass1}, we have $\varrho_0 \in \calH$, so that $n[\varrho_0] \in \H^2$ according to Lemma \ref{H2n}. Moreover, $n_0:=n[\varrho_0] > 0$, so that the moment problem with constraint $n_0$ admits a unique solution $\varrho_e[\varrho_0]=\exp -(H+A_0)$ in $\calE_+$, with, according to Theorem \ref{regA}, $A_0 \in L^2$. In turn, Proposition \ref{regmin} yields $\varrho_e[\varrho_0] \in \calH$. We can thus apply Lemma \ref{anton} and obtain the existence and uniqueness of the solution $\varrho_1\in\calC^0([0,T],\calH)$ to \fref{mild} for $k=0$. Since the propagator $\calL$ preserves positivity,  $\varrho_{1}$ and $\varrho_{1}^2$ defined by \fref{b3} are positive operators. By linearity of the trace, this yields
$$
n[\varrho_1](x)=n[\varrho^{1}](x)+n[\varrho_{1}^2](x) \geq n[\varrho^{1}](x), \qquad \forall x \in [0,1].
$$
Furthermore, using \fref{estW11}, Assumption \ref{ass1}, and the fact that $\calL(t)$ is an isometry on $\calE$, we find
$$
\|n[\calL(t) \delta \rho]\|_{L^\infty} \leq \|n[\calL(t) \delta \rho]\|_{W^{1,1}} \leq 2 \| \calL(t) \delta \rho \|_{\calE}=2 \|\delta \rho \|_{\calE}<\gamma.
$$
Finally, remarking that $\calL(t)f(H)=e^{-iHt}f(H)e^{iHt}=f(H)$, Assumption \ref{ass1} yields
\bee
e^{\frac{t}{\tau}}n[\varrho^1(t)]&=& n[\calL(t)f(H)]+n[\calL(t) \delta \rho]= n[f(H)]+n[\calL(t) \delta \rho]\\
&\geq & (1-\eps) \gamma,
\eee
for some $\eps \in (0,1)$, and therefore, $\forall t \in [0,T]$,
$$n[\varrho^1(t)]\geq (1-\eps) \gamma e^{-\frac{T}{\tau}} :=\underline{n}.
$$
This completes the initial step.
\bs

\paragraph{\it From step $k$ to $k+1$} Let now $\varrho_k \in L^\infty((0,T),\calH)$ with $n[\varrho_{k}(t)] \geq \underline{n}>0$ for all $t\in[0,T]$. Then, Lemma \ref{H2n} and Proposition \ref{regmin} imply that $\varrho_e[\varrho_k]$ belongs to $ L^\infty((0,T),\calH)$. We can then apply Lemma $\ref{anton}$ and obtain the existence  of a unique $\varrho_{k+1}$ in $\calC^0([0,T],\calH)$ solution to \fref{mild}. We have immediately the lower bound
$$
n[\varrho_{k+1}(t)] \geq n[\varrho^1(t)]\geq \underline{n}, \qquad \forall k \in \NN, \qquad \forall t\in[0,T].
$$
\bs

\paragraph{\it Uniform bounds} We start with uniform estimates in $\calJ_1$ and in $\calE$, that will in turn be used to obtain a uniform bound in $\calH$. The estimate \eqref{estianton} yields
\be \label{estant}
\| \varrho_{k+1} \|_{E} \leq e^{-\frac{t}{\tau}} \| \varrho^0\|_{E} +\frac{1}{\tau}\int_0^t e^{-\frac{t-s}{\tau}} \|\varrho_e [\varrho_k(s)]\|_{E}\, ds,
\ee
with $E=\calJ_1$ or $\calE$. Since by construction
$$
\|\varrho_e [\varrho_k(s)]\|_{\calJ_1}= \| n[\varrho_k(s)]\|_{L^1}=\|\varrho_k(s)\|_{\calJ_1},
$$
we find, iterating \fref{estant}, for all $t\in [0,T]$,
$$
\| \varrho_{k+1}(t) \|_{\calJ_1} \leq e^{-\frac{t}{\tau}}\| \varrho^0\|_{\calJ_1}\left(\sum_{p=0}^{k+1} \frac{(t/\tau)^p}{p!} \right)\leq \| \varrho^0\|_{\calJ_1}.
$$
This yields in particular $ \| n[\varrho_k(t)]\|_{L^1} \leq \| n[\varrho^0]\|_{L^1}$ for all $t\in [0,T]$, and, from \fref{estimsolmom1} and \fref{sqrt},
$$
 \|\varrho_e [\varrho_k(t)]\|_{\calE} \leq C \left(1+ \beta(\|n[\varrho_k(t)]\|_{L^1})+\| \sqrt{n[\varrho_k(t)]} \|^2_{\H^1}\right) \leq C+C \| \varrho_k(t)\|_{\calE}.
$$
Hence, according to \fref{estant},
$$
\| \varrho_{k+1}(t) \|_{\calE} \leq e^{-\frac{t}{\tau}}  \| \varrho^0\|_{\calE} + C \int_0^t  e^{-\frac{t-s}{\tau}}\big(1+\|\varrho_k(s)\|_{\calE}\big)\, ds,
$$
which yields by iterating, for all $t\in[0,T]$,
\be \label{estimrhoE}
\| \varrho_{k+1}(t) \|_{\calE} \leq e^{-\frac{t}{\tau}}\big(1+\| \varrho^0\|_{\calE}\big) \left(\sum_{p=0}^{k+1} \frac{(Ct)^p}{p!} \right)\leq e^{C T} (1+\| \varrho^0\|_{\calE}).
\ee
This provides a uniform bound in $\calC^0([0,T],\calE)$ for $\varrho_k$. Estimates \fref{estimsolmom2} and \fref{sqrt} then imply that $A_{k}(t)$ is bounded in $\H^{-1}_{per}$ by a constant independent of $k$ and $t$. Furthermore, the estimate of Theorem \ref{regA}, together with \fref{estimrhoE} and Lemma \ref{H2n}, yields, for all $t\in [0,T]$:
$$
\| A_k(t)\|_{L^2} \leq C \big(1+ \| \Delta n[\varrho_{k}(t)]\|_{L^2}\big) \leq C \big(1+ \big(\Tr \big(H\varrho_{k}(t) H \big) \big)^{1/2}\big),
$$
so that, according to Proposition \ref{regmin} with $a=0$ and $\alpha=1$,
\be \label{HrhoH}
\Tr \big(H \varrho_e[\varrho_{k}(t)]H \big) \leq C\big(1+\| A_k(t)\|^2_{L^2}\big) \leq C\big(1+\Tr \big(H\varrho_{k}(t) H \big)\big).
\ee
The estimate \eqref{estianton} with $E=\calH$ finally yields
$$
\|\varrho_{k+1}(t)\|_{\calH} \leq e^{-\frac{t}{\tau}} \| \varrho^0\|_{\calH} +C\int_0^t e^{-\frac{t-s}{\tau}} \big(1+\|\varrho_{k}(s)\|_{\calH}\big)\, ds,
$$
which gives, as above for \eqref{estimrhoE},
$$
\| \varrho_{k+1}(t) \|_{\calH} \leq e^{C T} (1+\| \varrho^0\|_{\calH}).
$$
Estimate \fref{estimrhoe} follows directly from  \fref{HrhoH}. This ends the proof of the proposition.

\end{proof}

Since $L^\infty((0,T),\calJ_1)$ is the dual of the space $L^1((0,T),\calK)$, the uniform bounds of Proposition \ref{unifbound} allow us to pass to the weak-$*$ limit in \fref{mild}, but not to identify the limit of the term $\varrho_e[\varrho_k]$. We need a stronger topology in order to take advantage of the H\"older estimate of Corollary \ref{corcont}. A first ingredient in this direction is the next proposition.
\begin{proposition} \label{hold} There exists $M^{''}$ independent of $k$ such that$$
\|\varrho_k\|_{\calC^{0,1/2}([0,T],\calE)} \leq M^{''},
$$
where $\calC^{0,1/2}([0,T],\calE)$ is the space of operators in $\calE$ with  H\"older continuity of order $1/2$ in the time variable. 
\end{proposition}
\begin{proof} The proof essentially relies from an interpolation argument. We will use the following two lemmas, that are proved in Sections \ref{proofhold0} and \ref{proofhold2}.
\begin{lemma} \label{hold0} Let $\varrho \in \calH$, self-adjoint and nonnegative. Then,
$$
\|\calL(t)\varrho-\varrho\|_{\calJ_1} \leq C t \|\varrho\|_{\calH}\qquad \mbox{for all }t\geq 0.
$$
\end{lemma}

\begin{lemma} \label{hold2} Let $\varrho \in \calC^0([0,T],\calH)$. Then, we have the estimate, for all $(t,s)$ in $[0,T] \times [0,T]$:
$$\| \varrho(t)-\varrho(s)\|_{\calE} \leq \sqrt{2} \|\varrho\|^{1/2}_{\calC^0([0,T],\calH)}\| \varrho(t)-\varrho(s)\|^{1/2}_{\calJ_1}.$$
\end{lemma}

We have all the tools to proceed to the proof of Proposition \ref{hold}. According to Lemma \ref{hold2}, we only need to estimate $\varrho_k(t)-\varrho_k(u)$ in $\calJ_1$ since it is proved in Proposition \ref{unifbound} that $\varrho_k$ is uniformly bounded in $\calC^0([0,T],\calH)$.  We write then, for $(t,s) \in [0,T]\times[0,T]$:
 \bee
\varrho_{k+1}(t)-\varrho_{k+1}(s)&=&e^{-\frac{t}{\tau}}\calL(t) \varrho^0-e^{-\frac{s}{\tau}}\calL(s) \varrho^0\\
&&+\frac{1}{\tau}\int^t_s e^{-\frac{t-\sigma}{\tau}}\calL(t-\sigma) \varrho_e[\varrho_{k}(\sigma)] d\sigma\\
&&+\frac{1}{\tau}\int^s_0 e^{\frac{\sigma}{\tau}}(e^{-\frac{t}{\tau}}\calL(t-\sigma)-e^{-\frac{s}{\tau}}\calL(s-\sigma)) \varrho_e[\varrho_{k}(\sigma)] d\sigma\\
& :=&T_1+T_2+T_3.
\eee
We estimate $T_1$ using Lemma \ref{hold0} and the fact that $\calL$ is  an isometry on $\calJ_1$ as follows:
\bee
\|T_1\|_{\calJ_1} &\leq &|e^{-\frac{t}{\tau}}-e^{-\frac{s}{\tau}}| \| \calL(t) \varrho^0\|_{\calJ_1}+\| \calL(s)(\calL(t-s)-\II) \varrho^0\|_{\calJ_1}\\
&\leq & C |t-s| \|\varrho^0\|_{\calJ_1}+ C |t-s| \|\varrho^0\|_{\calH}\\
& \leq & C |t-s|.
\eee
Regarding $T_2$, we have, thanks to \fref{estimrhoe} of Proposition \ref{unifbound},
\bee
\|T_2\|_{\calJ_1} &\leq & C |t-s| \| \varrho_e [\varrho_k]\|_{L^\infty((0,T),\calJ_1)} \leq C|t-s|.
\eee
Finally, using again Lemma \ref{hold0} and \fref{estimrhoe}, 
\bee
\|T_3\|_{\calJ_1} &\leq &C |e^{-\frac{t}{\tau}}-e^{-\frac{s}{\tau}}| \| \varrho_e[\varrho_k]\|_{L^\infty((0,T),\calJ_1)}+\|(\calL(t-s)-\II) \varrho_e[\varrho_k]\|_{L^\infty((0,T),\calJ_1)}\\
&\leq & C |t-s| \| \varrho_e[\varrho_k]\|_{L^\infty((0,T),\calJ_1)}+ C |t-s| \|\varrho_e[\varrho_k]\|_{L^\infty((0,T),\calH)}\\
& \leq & C |t-s|.
\eee
As a consequence, according to Lemma \ref{hold2}, for all $(t,s) \in [0,T] \times [0,T]$, we have
$$
\|\varrho_k(t)-\varrho_k(s)\|_{\calE} \leq C |t-s|^{1/2},
$$ 
which concludes the proof of the proposition.
\end{proof}

\subsection*{Step 2: Compactness} The next step is to turn the uniform estimates of Proposition \ref{unifbound} and Proposition \ref{hold} into compactness results. This is the object of the next proposition.
\begin{proposition} \label{compact} Let $T>0$ and let $(\varrho_k)_{k \in \NN}$ be defined by \fref{mild}. Then, there exist  $\varrho \in \calC^0([0,T],\calE_+) \cap  L^\infty((0,T),\calH)$ and a subsequence, still denoted by $(\varrho_k)_{k\in \NN}$, such that
$$
\varrho_k \to \varrho \qquad \textrm{ in } \calC^0([0,T],\calJ_1)
$$
and 
$$
\sqrt{H}\varrho_k \sqrt{H} \to \sqrt{H}\varrho\sqrt{H} \qquad \textrm{ in } \calC^0([0,T],\calJ_1).
$$

\end{proposition}
\begin{proof} We will show that the set 
\begin{align*}
&\calF=\{\varrho_k \in \calC^0([0,T],\calE_+), k \in \NN ,\; \textrm{such that } \|\varrho_k \|_{\calC^0([0,T],\calH)} \leq M,\\
 &\hspace{3cm} \textrm{and } \|\varrho_k\|_{\calC^{0,1/2}([0,T],\calE)} \leq M^{''} \} \subset \calC^0([0,T],\calE_+),
\end{align*}
 is equicontinuous, and pointwise relatively compact. The equicontinuity is a consequence of the uniform bound in $\calC^{0,1/2}([0,T],\calE)$. In order to prove that $\calF$ is pointwise relatively compact, we need to show that for each $t \in [0,T]$, the set $\calF_t=\{\varrho_k(t), \textrm{ such that }\varrho_k \in \calF,\, k\in \NN  \}$ is relatively compact in $\calE_+$, which is equivalent to the fact that any sequence in $\calF_t$ admits a subsequence converging in $\calE_+$. This will follow from Lemma \ref{strongconv} of the Appendix. 
 
 First of all, for each fixed $t\in[0,T]$, $\varrho_k(t)$ is bounded in $\calE_+$, so that according to Lemma \ref{strongconv}, there exists $\varrho(t) \in \calE_+ $, such that, up to an extraction of a subsequence,
$$
\varrho_k(t) \to \varrho(t) \qquad \textrm{strongly in } \calJ_1.
$$
In the same way, $\sqrt{H}\varrho_k(t)\sqrt{H}$ is bounded in $\calE_+$, so there exists $\widetilde{\varrho}(t) \in \calE_+ $, such that, up to an extraction of a subsequence,
$$
\sqrt{H}\varrho_k(t) \sqrt{H}\to \widetilde{\varrho}(t) \qquad \textrm{strongly in } \calJ_1.
$$
It is not difficult to identify $\widetilde{\varrho}(t)$ with $\sqrt{H}\varrho(t)\sqrt{H}$. This proves the pointwise relative compactness. The Arzel\`a-Ascoli theorem then yields that $\calF$ is relatively compact in $\calC^0([0,T],\calE_+)$. Therefore, since $(\varrho_k)_{k \in \NN}$ belongs to $\calF$ according to the uniform bounds of Proposition \ref{unifbound} and Proposition \ref{hold}, we can extract a subsequence such that $\varrho_k$ and $\sqrt{H}\varrho_k\sqrt{H}$ converge to $\varrho$ and $\sqrt{H}\varrho\sqrt{H}$ in $\calC^0([0,T],\calJ_1)$. The fact that $\varrho \in L^\infty((0,T),\calH)$ follows the uniform bound in $\calC^0([0,T], \calH)$ and from the compactness of $(\varrho_k)_{k \in \NN}$ for the weak-$*$ $L^\infty((0,T),\calH)$ topology.
\end{proof}

 A direct corollary of the previous proposition and of Lemma \ref{LT}, stated without proof, is the following result, that will help in the identification of the limit of $\varrho_e[\varrho_k]$. 
\begin{corollary} \label{convn} With the notation of Proposition \ref{compact}, $n[\varrho_k]$ converges strongly in $\calC^0([0,T],\H^1)$ to $n[\varrho]$.
\end{corollary}
 
\subsection*{Step 3: Passing to the limit} We pass now to the limit in \fref{mild}. First of all, since $\varrho_e[\varrho_k]$ is uniformly bounded in $L^\infty((0,T),\calH)$ according to \fref{estimrhoe}, there exists $\widetilde{\varrho} \in L^\infty((0,T),\calH)$ such that $\varrho_e[\varrho_k]$ converges to $\widetilde{\varrho}$ in $L^\infty((0,T),\calH)$ weak-$*$. Together with Proposition \ref{compact}, this allows us to take the limit in \eqref{decomprhok}, \eqref{b3} and to obtain, for all $t\in[0,T]$,
\be \label{eqlim}
\varrho(t)=e^{-\frac{t}{\tau}}\calL(t) \varrho^0+\frac{1}{\tau}\int_0^t e^{-\frac{t-s}{\tau}}\calL(t-s) \widetilde{\varrho}(s) ds.
\ee
Note that actually $\varrho \in \calC^0([0,T],\calH)$ since $\widetilde{\varrho} \in L^\infty((0,T),\calH)$ in \fref{eqlim}. The next step is to identify $\widetilde{\varrho}$. We will apply Corollary \ref{corcont} for this identification. 

We first need to make sure that the moment problem with density constraint $n[\varrho(t)]$ can be solved. This follows from the fact that $\varrho(t) \in \calH$ for all $t \in[0,T]$, and from $n[\varrho(t)] \geq \underline{n}$, since $n[\varrho_k]$ converges to $n[\varrho]$ in $\calC^0([0,T],\H^1)$ (and therefore in $\calC^0([0,T] \times [0,1])$) according to Corollary \ref{convn}.  

Denote then by $\varrho_e[\varrho(t)]$ the solution to the moment problem with constraint $n[\varrho(t)]$ for every $t \in [0,T]$. Since $\varrho \in \calC^0([0,T],\calH)$, there exists $M'>0$ such that $\|\varrho(t)\|_{\calH} \leq M'$ for all $t\in[0,T]$. Together with \fref{estimkH}, this shows that the hypotheses of Corollary \ref{corcont} are satisfied. Therefore,
$$
\|\varrho_e[\varrho_k(t)]-\varrho_e[\varrho(t)]\|_{\calJ_2} \leq C \| \varrho_k(t) -\varrho(t)\|^{ 1/8}_{\calJ_2},
$$
which shows that $\widetilde{\varrho}=\varrho_e[\varrho]$ according to Proposition \ref{compact}.
\bs

\subsection*{ Step 4: Conclusion} We have proved that the equation
$$
\varrho(t)=e^{-\frac{t}{\tau}}\calL(t) \varrho^0+\frac{1}{\tau}\int_0^t e^{-\frac{t-s}{\tau}}\calL(t-s) \varrho_e[\varrho(s)] ds
$$
admits a solution $\varrho \in \calC^0([0,T],\calH)$.  Let us prove that $\varrho$ is a classical solution to \fref{liouville2}. We just give a sketch of the proof since the arguments are classical.  On the one hand, we first remark that since $\varrho^0 \in \calH$, $\hat \varrho(t):=e^{-\frac{t}{\tau}}\calL(t) \varrho^0 \in \calC^0([0,T],\calH)$ is continuously differentiable, with values in $\calJ_1$, and that $\hat \varrho$ satisfies
$$
\frac{\partial}{\partial t } \hat \varrho(t) =L_0(\hat \varrho(t)) -\frac{1}{\tau} \hat \varrho(t).
$$
On the other hand, according to \cite{pazy}, Chapter 4, Theorem 2.4, we need to show in addition that $\varrho_e[\varrho] \in \calC^0([0,T],\calJ_1)$ and that
$$
v(t):=\int_0^t e^{-\frac{t-s}{\tau}}\calL(t-s) \varrho_e[\varrho(s)] ds
$$
is continuously differentiable in $\calJ_1$. We start with the continuity of $\varrho_e[\varrho(t)]$. The continuity in $\calJ_2$ is a consequence of Corollary \ref{corcont} and of the continuity of $\varrho$ in $\calH$. 

The continuity in $\calJ_1$ follows from compactness arguments. Indeed, pick a time $t$ in $[0,T]$ and let $t_\eps$ such that $t_\eps \to t$. Let $n=n(t)=n[\varrho(t)]=n[\varrho_e[\varrho(t)]]$, as well as $n_\eps=n(t_\eps)$. We introduce similarly $\varrho_e^\eps=\varrho_e(t_\eps)$ and $\varrho_e=\varrho_e(t)$ to ease notations. We will prove that $\varrho_e^\eps$ converges to $\varrho_e$ in $\calJ_1$. First, since $\varrho \in \calC^0([0,T],\calH)$, the sequence $\varrho(t_\eps)$ is uniformly bounded in $\calH$. Consequently, according to Lemma \ref{H2n}, the sequence $n_\eps$ is bounded in $\H^2$. Therefore, following Theorem \ref{regA}, the sequence $A_\eps$ is uniformly bounded in $L^2$, and finally, according to Proposition \ref{regmin}, the sequence $\varrho_e^\eps$ is bounded in $\calH$ .  Lemma \ref{strongconv} then shows that there exists $\widetilde \varrho_e \in \calJ_1$ such that $\varrho_e^\eps \to \widetilde \varrho_e$ strongly in $\calJ_1$. We conclude by identifying $\widetilde \varrho_e$ with $\varrho_e(t)$ thanks to the continuity in $\calJ_2$. Hence, we have just proved that $\varrho_e[\varrho] \in \calC^0([0,T],\calJ_1)$.

We turn now to the differentiability of $v$. We write first
\begin{align*}
\frac{v(t+h)-v(t)}{h}&=\frac{1}{h}\int^{t+h}_t e^{-\frac{t+h-s}{\tau}}\calL(t+h-s) \varrho_e[\varrho(s)] ds\\
&\qquad+\frac{1}{h}\int^t_0 e^{\frac{s}{\tau}}(e^{-\frac{t+h}{\tau}}\calL(t+h-s)-e^{-\frac{t}{\tau}}\calL(t-s)) \varrho_e[\varrho(s)] ds\\
&=T_1(h)+T_2(h).
\end{align*}
{}From the continuity of $\varrho_e[\varrho]$ in $\calJ_1$, from the fact that $\calL$ is a $\calC^0$ unitary group on $\calJ_1$, one deduces that the limit in $\calJ_1$ of $T_1(h)$ as $h\to 0$ is equal to $\varrho_e[\varrho(t)]$. For the term $T_2$, we use the fact that $\varrho_e[\varrho](s) \in \calH$ for all $s \in [0,T]$.  This yields the differentiability of $v$, and that, for all $t \in [0,T]$,
$$
\frac{\partial}{\partial t} v(t)= L_0(v(t))-\frac{1}{\tau} v(t) +\varrho_e[\varrho(t)].
$$
It only remains to establish the continuity of $\partial_t v$ in $\calJ_1$, and therefore that of $L_0(v(t))$. Since $\frac{1}{\tau} v(t)=\varrho(t)-\hat \varrho(t)$ by definition, and both $\varrho$ and $\hat \varrho$ are in $\calC^0([0,T],\calH)$, it follows that $v \in \calC^0([0,T],\calH)$, which is enough to conclude that $L_0(v(t)) \in \calC^0([0,T],\calJ_1)$. This concludes the proof of Theorem \ref{thliou}.

\section{Proof of Theorem \ref{thliou2}} \label{proofthliou2}

We will need the following lemma, which is a direct adaptation of the results of \cite{DFM}, Section 4, and provides us with a solution to the moment problem with a \textit{global} density constraint (see as well \cite{Dolbeault-Loss}):
\begin{lemma} \label{momN} (The global moment problem) For some $n_0\in \RR_+^*$, let $\calE_0=\{ \varrho \in \calE_+, \; \Tr \varrho=n_0\}$. Then, the free energy  $F$ defined by \eqref{F} admits  a unique minimizer $\varrho_g$ in $\calE_0$, which admits the expression
$$
\varrho_g=e^{-(H+A_0)}, 
$$ 
where the chemical potential $A_0$ is a constant and verifies $n_0= \Tr(e^{-H}) e^{-A_0}$.
\end{lemma}
The proof of Theorem \ref{thliou2} involves the relative entropy between two density operators $u$ and $v$ defined by
$$
S(u,v)=\Tr \big((\log u-\log v) u \big).
$$
Note that for all $(u,v)$, $S(u,v) \in [0, \infty]$, and by Theorem \ref{klein}, for two density operators $u$ and $v$ with the same trace, we have $S(u,v)=0$ if and only if $u=v$. We will use the notation $\varrho_e(\tau):=\varrho_e[\varrho(\tau)]$. The key ingredient of the proof is the following lemma (proved in Section \ref{secother}), which shows that the free energy is not increasing.
\begin{lemma} (Entropy relation) \label{eqfreeE}For $t,s \geq 0$, for $\varrho$ a classical solution to \fref{liouville2}, and $\varrho_e$ the solution to the moment problem with density constraint $n[\varrho(t)]$, let 
$$
\calS(t,s)=\frac{1}{\tau}\int_s^t \left(S(\varrho(\sigma),\varrho_e(\sigma))+S(\varrho_e(\sigma),\varrho(\sigma))\right) d\sigma.
$$
Then, the free energy verifies the relation
\be \label{eqentropie}
F(\varrho(t))+\calS(t,s)=F(\varrho(s)), \qquad \forall t,s \geq 0.
\ee
\end{lemma}
The proof of Theorem \ref{thliou2} then goes formally as follows. Let $(t_k)_{k \in \NN}$ be a sequence such that $t_k \to \infty$ and define $\varrho_k(t):=\varrho(t+t_k)$ for $t\in [0,T]$. Here $T>0$ is a fixed constant. In a first step, we show that $\varrho_k(t)$ converges to a \textit{local} quantum Maxwellian (i.e. a solution to the moment problem with a local density constraint); this is a consequence of entropy dissipation via \fref{eqentropie}, and of continuity results for the moment problem that enable us to identify the limit of $\varrho_e[\varrho_k(t)]$. In a second step, we use the free Liouville equation to show that the local Maxwellian is actually a global Maxwellian: since $\varrho_k$ converges to a local Maxwellian $\varrho_\infty(t):=\exp(-(H+A_\infty(t)))$, the Liouville-BGK equation becomes a free Liouville equation of the form
$$
i\partial_t \varrho_\infty=[H,\varrho_\infty].
$$
Then, writing 
$$
[H,\varrho_\infty]=[H+A_\infty(t),\varrho_\infty]-[A_\infty(t),\varrho_\infty]=-[A_\infty(t),\varrho_\infty],
$$
where the first commutator vanishes since $\varrho_\infty(t)=\exp(-(H+A_\infty(t)))$, we obtain that $n[\varrho_\infty(t)]$ is constant in time since the local trace of the second commutator is zero. Therefore, since $\varrho_\infty(t)$ is the solution to the moment problem with constraint $n[\varrho_\infty(t)]$, $\varrho_\infty(t)$ is itself independent of time. As a consequence, $[H,\varrho_\infty]=0$, so that $H$ and $\varrho_\infty$ can be diagonalized simultaneously. In particular, since the eigenspace of $H$ corresponding to the zero eigenvalue is of dimension one, the corresponding eigenvector, namely the constant one, is also an eigenvector of $H+A_\infty(t)$, say associated with the eigenvalue $\lambda_{q_0}$. Denoting by $\phi_1=1$ this eigenvector, we have as a consequence,
$$
(H+A_\infty(t)) \phi_1 = \lambda_{q_0} \phi_1= A_\infty(t) \phi_1,
$$ 
which shows that for all $t\geq 0$, $A_\infty(t,x)$ is a constant (note that we use here the particular form of the free Hamiltonian $H$; if $H$ included a potential term, we would resort to the Krein-Rutman theorem to show that $A_\infty(t,x)$ is a constant). Hence, for all $t\geq 0$, $\varrho_\infty(t)$ reads $\varrho_\infty(t)=e^{-(H+\lambda_{q_0})}$, verifies moreover $\Tr \varrho^0=\Tr \varrho_\infty(t)= \Tr(e^{-H}) e^{-\lambda_{q_0}}$, which, according to Lemma \ref{momN}, shows that  $\varrho_\infty(t)=\varrho_g$, with $n_0=\Tr \varrho^0$. 

The rigorous proof starts with some estimates.

\subsection{Estimates} Since $\calS(t,s) \geq 0$ for all $t,s \geq 0$, we have from \fref{eqentropie},
\be \label{decF} 
F(\varrho(t))\leq F(\varrho^0), \qquad \forall t \geq 0.
\ee
Using the fact that the Liouville equation preserves the trace, which yields $\Tr \varrho(t)= \Tr \varrho(0)$, and using \fref{decF} along with \fref{souslin}, we can conclude that 
\be \label{rhon}
\sup_{t \in [0,\infty)} \|\varrho_k(t)\|_{\calE} \leq C, \qquad \forall k \in \NN.
\ee
In turn, we deduce from \fref{sqrt} and \fref{estimsolmom1} that
\be \label{rhoen}
\sup_{t \in [0,\infty)} \|\varrho_e[\varrho_k(t)]\|_{\calE} \leq C, \qquad \forall k \in \NN.
\ee
We need more estimates in order to guarantee sufficient compactness. We adapt Proposition \ref{hold} and Lemma \ref{hold0} to the case where $\varrho(t)$ is only in $\calE$ and not in $\calH$. The estimate of Lemma \ref{hold0} then becomes, for $\varrho \in \calE$, nonnegative and self-adjoint,  
\be \label{hold00}
\|\calL(t)\varrho-\varrho\|_{\calJ_1} \leq C \sqrt{t} \|\varrho\|_{\calE}\qquad \mbox{for all }t\geq 0.
\ee
A proof of the above estimate is given in Section \ref{proofhold00} for completeness. Moreover, following step by step the proof of Proposition \ref{hold} and using \fref{hold00},  we obtain, for all $s,t$ in $[0,T]$,
\be \label{hold3}
\|\varrho_k(t)-\varrho_k(s)\|_{\calJ_1} \leq C |t-s|^{1/2}  \left( \|\varrho_k(0)\|_{\calE}+ \|\varrho_e[\varrho_k]\|_{L^\infty((0,\infty),\calE)} \right) \leq C |t-s|^{1/2}.
\ee
In the latter estimate, we used \eqref{rhon} and \fref{rhoen} and the constant $C$ does not depend on $k$ (but depends on $T$). Estimates \fref{rhon} and \fref{hold3} are then enough to conclude that, in the same way as in Proposition \ref{compact}, there exists $\varrho_\infty \in \calC^0([0,T],\calJ_1) \cap L^\infty((0,T),\calE)$, nonnegative, and a subsequence, still denoted by $(\varrho_k)_k$ such that, as $k \to \infty$,
\be \label{convinfty}
\varrho_k \to \varrho_\infty \qquad \textrm{strongly in } \calC^0([0,T],\calJ_1).
\ee

\subsection{Convergence to a local quantum Maxwellian} We identify now the limit $\varrho_\infty$. For this, we need to characterize the solution to the moment problem with constraint $n[\varrho_\infty(t)]$.  This requires $n[\varrho_\infty(t)]$ to be uniformly bounded from below, and will allow us to exploit continuity results for the minimizer. Note that the lower bound we obtained in the existence theory is useless here since it vanishes in the limit $t_k \to \infty$.

\subsection*{Step 1: bound from below} We obtain a new bound as follows: denote by $(\mu_p,e_p)_{p \in \NN^*}$ the spectral elements of the Hamiltonian $H+A_0$, with domain $D(H)$ explicited in \fref{domainH}, that defines the Gibbs state $\varrho_g$ of Lemma \ref{momN}. The eigenfunctions of $H+A_0$ are the Fourier basis $e^{2\pi i px}$, $p\in \ZZ$. Hence, the density of the Gibbs state is constant since
\be \label{lowerg}
n[\varrho_g](x)\equiv n_g:=e^{-A_0}\sum_{p\in \ZZ} e^{-4\pi^2p^2}>0.
\ee
We show now that if the relative entropy between the initial condition $\varrho^0$ and $\varrho_g$ is sufficiently small, then the difference $n[\varrho_\infty(t)]-n[\varrho_g]$ remains small as well in $L^\infty$. This will enable us to exploit \fref{lowerg} in order to obtain the bound from below.
Knowing that $\varrho_g=e^{-(H+A_0)}$, $A_0$ being constant, and using that $\varrho_k(t) \in \calH$ for all $t \geq 0$, we write
\bee
F(\varrho_k(t))&=& \Tr \big (H \varrho_k(t)\big) + \Tr \big( (\log \varrho_k(t)-\II) \varrho_k(t) \big)\\
&=& -\Tr \big (\log (\varrho_g) \; \varrho_k(t)\big)- A_0 \Tr \big( \varrho_k (t)\big)  + \Tr \big( (\log \varrho_k(t)-\II) \varrho_k(t) \big)\\
&=&-(A_0+1) \Tr \big (\varrho_k(t)\big)+\Tr \big( (\log \varrho_k(t)-\log \varrho_g ) \varrho_k(t) \big)\\
&=&F(\varrho_g)+S(\varrho_k(t),\varrho_g),
\eee 
where we used the fact that $\Tr \big (\varrho_k(t)\big)=\Tr \big (\varrho^0\big)=\Tr \big (\varrho_g\big)$ in the last line. We then deduce from \fref{decF} that
$$
S(\varrho_k(t),\varrho_g) \leq S(\varrho^0,\varrho_g),  \qquad \forall t \geq 0, \qquad \forall k \in \NN.
$$
The Klein inequality of Theorem \ref{klein} then yields
$$
\| \varrho_k(t)-\varrho_g\|_{\calJ_2}^2 \leq C S(\varrho^0,\varrho_g), \qquad \forall t \geq 0, \qquad \forall k \in \NN,
$$
which, with \fref{convinfty} and the fact that $\calJ_1 \subset \calJ_2$, leads to
\be \label{relent}
\| \varrho_\infty(t)-\varrho_g\|_{\calJ_2}^2 \leq C S(\varrho^0,\varrho_g), \qquad \forall t \in [0,T].
\ee
Finally, estimate \fref{ninfty} gives
$$
\| n[\varrho_\infty(t)] - n[\varrho_g] \|_{L^\infty} \leq C \| \varrho_\infty(t)-\varrho_g\|_{\calJ_2}^{1/4} \left( \|\varrho_\infty(t)\|_{\calE}+\| \varrho_g\|_\calE\right)^{3/4}.
$$
Combining the latter estimate with \fref{relent} and the fact that $\varrho_\infty \in L^\infty((0,T),\calE)$, we obtain
$$
\| n[\varrho_\infty] - n[\varrho_g] \|_{L^\infty((0,T) \times (0,1))} \leq C_0 (S(\varrho^0,\varrho_g))^{1/8} =C_0 (F(\varrho^0)-F(\varrho_g))^{1/8}.
$$
Setting then $F(\varrho^0)-F(\varrho_g)\leq (\eps_0)^8$, with $C_0 \eps_0<n_g$, we can conclude from \fref{lowerg} that $n[\varrho_\infty(t)]$ is bounded from below a.e. in $(0,T) \times (0,1)$. 

\subsection*{Step 2: solution to the limiting moment problem} Since $\varrho_\infty \in L^\infty((0,T),\calE)$, we have from \fref{ninfty}-\fref{gradnl2} that $n[\varrho_\infty(t)] \in \H^1$, $t$ a.e., and we can therefore solve the moment problem for the constraint $n[\varrho_\infty(t)]$. Note that since we only know that $\varrho_\infty(t)$ is in $\calE$ almost everywhere in $[0,T]$ and not for all $t\in[0,T]$, the solution to the moment problem with constraint $n[\varrho_\infty(t)]$ is only defined almost everywhere in $[0,T]$.  We denote the corresponding solution by $\varrho_e[\varrho_\infty(t)]$. Since we just showed in the previous step that $n[\varrho_\infty(t)]$ is uniformly bounded from below, we can characterize $\varrho_e[\varrho_\infty(t)]$ as a quantum Maxwellian with chemical potential $A_\infty(t)$. Following \fref{estimsolmom1}-\fref{estimsolmom2}, we have moreover the estimates
\be \label{rhoen2}
\|\varrho_e[\varrho_\infty]\|_{L^\infty((0,T),\calE)} \leq C, \qquad \|A_\infty\|_{L^\infty((0,T),\H^{-1}_{per})} \leq C.
\ee
Note that the same analysis yields that $n[\varrho_k(t)]$ is uniformly bounded from below at all times, which, using \fref{estimsolmom2} and \fref{rhon}, leads to 
\be \label{rhoAn2}
\sup_{t \in [0,\infty)} \|A_k(t)\|_{\H^{-1}_{per}} \leq C, \qquad \forall k\in \NN.
\ee

\subsection*{Step 3: continuity of quantum Maxwellians for $\H^{-1}_{per}$ potentials} The third step in the identification of $\varrho_\infty$ is to use once more the entropy relation \fref{eqentropie}, now for the integral term involving the relative entropies, to arrive at 
\begin{align*}
\frac{1}{\tau}\int_0^{t+t_k} S(\varrho(s),\varrho_e(s)) ds &\leq F(\varrho^0)-F(\varrho(t+t_k))\\
&\leq F(\varrho^0)+C\left(\Tr(\sqrt{H}\varrho_k(t)\sqrt{H})\right)^{1/2}\leq C
\end{align*}
where we used \eqref{souslin} and \eqref{rhon}. Hence, since $S(\varrho(s),\varrho_e(s))$ is nonnegative, one deduces that
$$\forall t\in [0,T],\qquad \lim_{k\to \infty}\int_0^{t+t_k} S(\varrho(s),\varrho_e(s)) ds=\int_0^{+\infty} S(\varrho(s),\varrho_e(s)) ds<+\infty.$$
Therefore, one gets from Theorem \ref{klein} that, for any $T \in (0,\infty)$,
\begin{align*}  \nonumber
&\lim_{k\to \infty} \int_{0}^T \| \varrho(s+t_k)-\varrho_e(s+t_k)\|^2_{\calJ_2} d s \\ \nonumber
&\qquad \leq   C\lim_{k\to \infty} \int_{t_k}^{t_k+T} S(\varrho(s),\varrho_e(s)) d s\\
&\qquad =C\lim_{k\to \infty} \int_0^{t_k+T} S(\varrho(s),\varrho_e(s)) d s-C\lim_{k\to \infty} \int_0^{t_k} S(\varrho(s),\varrho_e(s)) d s=0.
\end{align*}
 We have proved that
\begin{equation}
\label{convintrho}
\lim_{k\to \infty} \int_{0}^T \| \varrho_k(s)-\varrho_e[\varrho_k(s)]\|^2_{\calJ_2}=0
\end{equation} 
If we can characterize the limit of $\varrho_e[\varrho_k(s)]$, then we will be able to identify $\varrho_\infty$ using \fref{convinfty}. For this, we will exploit the continuity of the minimizer with respect to the constraint. When the Lagrange parameter $A_\infty(t)$ is in $L^2$, then the continuity follows from Proposition \ref{contsol} and Corollary \ref{corcont}. Unfortunately, since $\varrho_\infty(t)$ is only in $\calE$ at this point, we only know that $A_\infty(t)$ is in $\H^{-1}_{per}$ and not in $L^2$, and Proposition \ref{contsol} cannot be used as stated. We generalize it to potentials in $\H^{-1}_{per}$ as follows: take two potentials $A_1$ and $A_2$ in $\H^{-1}_{per}$, with regularizations $A_1^\eps$ and $A_2^\eps$ in $\calC^\infty([0,1])$ that converge strongly in $\H^{-1}_{per}$ to $A_1$ and $A_2$. Let $H_{A_1^\eps}=H+A_1^\eps$ and $H_{A_2^\eps}=H+A_2^\eps$ be self-adjoint operators with domain $D(H)$. According to Remark \ref{rem}, we have from Proposition \ref{contsol},
\be \label{estimreg}
\| \exp (-H_{A_1^\eps})- \exp(-H_{A_2^\eps})\|_{\calJ_2}^2 \leq C \big(A^\eps_2-A^\eps_1,n^\eps_1-n^\eps_2\big),
\ee
where $C$ is independent of $A_1^\eps$ and $A_2^\eps$, and $n_1^\eps$ and $n_2^\eps$ are the local densities of $\exp(-H_{A_1^\eps})$ and $\exp(-H_{A_2^\eps})$. We then pass to the limit in \fref{estimreg}. Consider the following quadratic forms, defined on $\H^1_{per}$,
\be \label{qquad}
Q_{A}(\varphi,\psi)=(\nabla \varphi,\nabla \psi)+(A, \overline{\varphi} \psi)_{\H^{-1}_{per},\Hunper},
\ee
for $A=A_1, A_2$. These forms are closed and semibounded, and according to Theorem VIII.15 of \cite{RS-80-I}, they are uniquely associated to self-adjoint operators that we denote by $H_{A_1}$ and $H_{A_2}$. In the same way, the operators $H_{A_1^\eps}$ and $H_{A_2^\eps}$ are associated to quadratic forms that we denote by  $Q_{A_1^\eps}$ and $Q_{A_2^\eps}$. We have then the following estimate, for $j \in \{1,2\}$,
\bea \nonumber
|Q_{A_j}(\varphi,\varphi)-Q_{A_j^\eps}(\varphi,\varphi)| &\leq& \|A_j-A_j^\eps\|_{\H^{-1}_{per}}  \| |\varphi|^2\|_{\H^1}\\ \nonumber
&\leq& C \|A_j-A_j^\eps\|_{\H^{-1}_{per}}  \left( \| \varphi\|^2_{L^4}+\|\varphi\|_{L^\infty} \|\nabla \varphi\|_{L^2} \right)\\ 
&\leq &C \|A_j-A_j^\eps\|_{\H^{-1}_{per}} \left(\left\|\nabla \varphi \right\|_{L^2}^2+\|\varphi\|^2_{L^2} \right), \label{convQ}
\eea
where we used the Gagliardo-Nirenberg inequality \fref{gag}. Theorem 3.6 of \cite{kato}, Chapter 6, shows that $H_{A_j^\eps}$ converges to $H_{A_j}$ in the generalized sense. This in turn implies, according to \cite{kato}, Theorem 2.25, Chapter 4, that the corresponding resolvents converge in $\calL(L^2)$. Finally, since $H_{A_j^\eps}$ is bounded from below uniformly in $\eps$, Theorem VIII.20 of \cite{RS-80-I} then yields that $\exp (-H_{A_j^\eps})$ converges to $\exp (-H_{A_j})$ in $\calL(L^2)$. We show below that the convergence actually holds in $\calJ_1$. We remark first that thanks to \fref{controlmu}-\fref{estimgrad},
\be \label{estimHA}
\|\exp (-H_{A_j^\eps})\|_{\calE} \leq C \exp(C \|A_j^{\eps}\|^4_{\H^{-1}_{per}})\leq C.
\ee
Let then $K$ be a compact operator, and let $P_m$ a sequence of increasing finite-dimensional orthogonal projections on $L^2$ such that $P_m \to \un$ strongly in $\calL(L^2)$. We have
$$
\Tr (\delta \varrho^\eps K)=\Tr ((\un-P_m) \delta \varrho^\eps K)+\Tr (P_m \delta \varrho^\eps K):= T^\eps_1+T^\eps_2,
$$
where
$$
 \delta \varrho^\eps:=\exp (-H_{A_j^\eps})-\exp (-H_{A_j}).
$$
Using \fref{estimHA}, we estimate $T^\eps_1$ by
$$
|T_1^\eps| \leq \|\un-P_m\|_{\calL(L^2)} \|\delta \varrho^\eps \|_{\calJ_1} \|K\|_{\calK} \leq C \|\un-P_m\|_{\calL(L^2)} \|K\|_{\calK} \leq \eta \|K\|_{\calK}
$$
for $\eta>0$ arbitrary and $m \equiv m(\eta)$ sufficiently large. For this very $m$, using the strong convergence of $\delta \varrho^\eps$ to zero in $\calL(L^2)$, we estimate $T_2^\eps$ by
$$
|T_2^\eps| \leq \|P_m\|_{\calJ_1} \|\delta \varrho^\eps \|_{\calL(L^2)} \|K\|_{\calK} \leq C_m \|\delta \varrho^\eps \|_{\calL(L^2)}  \|K\|_{\calK} \leq \eta \|K\|_{\calK},
$$ 
for $\eps$ sufficiently small. We then obtain the strong convergence in $\calJ_1$ by a duality argument,  using that
$$
\| u\|_{\calJ_1} = \sup_{\|K\|_\calK\leq 1} |\Tr (u K)|.
$$
In order to conclude, we remark that \fref{ninfty} and \fref{gradnl2}, together with \fref{estimHA} and $\delta \varrho^\eps \to 0$ in $\calJ_1 \subset \calJ_2$, show that $n_j^\eps$ converges to $n_j$ strongly in $\H^1$. This finally allows us to pass to the limit in \fref{estimreg} and to obtain
\be \label{estimreg2}
\| \exp (-H_{A_1})- \exp(-H_{A_2})\|_{\calJ_2}^2 \leq C \big(A_2-A_1,n_1-n_2\big)_{\H^{-1}_{per},\Hunper}
\ee
for $A_j \in \H^{-1}_{per}$.

\subsection*{Step 4: conclusion} We have everything needed now to identifty $\varrho_\infty$. First of all, from \fref{rhon}, \fref{convinfty}, and the facts that $\varrho_\infty \in L^\infty((0,T),\calE)$ and $\calJ_1 \subset \calJ_2$, we conclude from \fref{ninfty}-\fref{gradnl2} that $n[\varrho_k]$ converges to $n[\varrho_\infty]$ strongly in $L^\infty((0,T),\H^1)$. Applying then \fref{estimreg2} with $A_1=A_\infty(t)$, $A_2=A_k(t)$, $n_1=n[\varrho_\infty(t)]$, and $n_2=n[\varrho_k(t)]$, and using the uniform bound on $A_k$ given in \fref{rhoAn2}, we deduce that $\exp(-H_{A_k(t)})=\varrho_e[\varrho_k(t)]$ converges to $\exp(-H_{A_\infty(t)})=\varrho_e[\varrho_\infty(t)]$ in $L^\infty((0,T),\calJ_2)$. The convergence results \fref{convinfty} and \fref{convintrho} then allow us to conclude that $\varrho_\infty(t)=\varrho_e[\varrho_\infty(t)]=\exp(-H_{A_\infty(t)})$ a.e. in $[0,T]$, and therefore that $\varrho_\infty$ is a local quantum Maxwellian a.e. in $[0,T]$. 

In the next section, we will use the facts that, $t$ a.e., 
\be \label{relQH}
Q_{A_\infty(t)}(\varphi,\psi)=(\varphi,H_{A_\infty(t)} \psi)=(H_{A_\infty(t)} \varphi,\psi), \qquad \forall \varphi,\psi \in D(H_{A_\infty(t)}) \subset \H^1_{per},
\ee
and that, for all $\varphi \in L^2$,
\be \label{rhoD}
\varrho_\infty(t) \varphi \in D(H_{A_\infty(t)}) \qquad \textrm{ with } \qquad  \varrho_\infty(t) \varphi \in \H^1_{per}.
\ee
We have in particular the following estimate, denoting by $(\rho_p,\phi_p)_{p \in \NN^*}$ the spectral elements of $\varrho_\infty$,
\bea \nonumber
 \|\varrho_\infty(s) \varphi\|_{\H^1} &\leq& \left( \sum_{p \in \NN^*} (\rho_p(s))^2 \|\phi_p(s)\|_{\H^1}^2 \right)^{1/2} \|\varphi\|_{L^2}\\ \nonumber
&\leq & \left(\sup_{p \in \NN^*} |\rho_p(s)|\right)^{1/2}\|\varrho_\infty\|^{1/2}_{L^\infty((0,T),\calE)}\|\varphi\|_{L^2}\\
&\leq & C (\Tr \varrho_\infty(s))^{1/2} \|\varrho_\infty\|^{1/2}_{L^\infty((0,T),\calE)}\|\varphi\|_{L^2}<+\infty. \label{rhoH1}
\eea
\subsection{From the local to the global quantum Maxwellian} The analysis would be greatly simplified if we knew that $A_\infty(t) \in L^2$, but since we only have  $A_\infty(t) \in \H^{-1}_{per}$, some additional technicalities are necessary. We first pass  to the limit in the Liouville equation: let $\varphi$ and $\psi$ be two functions in $\calC^\infty_{per}([0,1])$, then, for all $t \geq 0$,
\bee
(\varrho_k(t) \varphi, \psi)&=&(\varrho_k(0) \varphi, \psi)-i\int_0^t \big[(\varrho_k(s) \varphi,H \psi)-(H \varphi, \varrho_k(s) \psi) \big] ds\\
&&+\frac{1}{\tau}\int_0^t ((\varrho_e[\varrho_k(s)]-\varrho_k(s))\varphi,\psi)ds.
\eee
Using \fref{convinfty} and \fref{convintrho}, we find, for all $t\in[0,T]$,
\bea \label{freeL}
(\varrho_\infty(t) \varphi, \psi)&=&(\varrho_\infty(0) \varphi, \psi)-i\int_0^t \big[(\varrho_\infty(s) \varphi,H \psi)-(H \varphi, \varrho_\infty (s)\psi) \big] ds.
\eea

\subsection*{Step 1: rewriting the free Liouville equation} We exploit here the fact that $\varrho_\infty$ is a quantum Maxwellian to simplify the free Liouville equation \fref{freeL}.  Consider $A_\infty \in L^\infty((0,T),\H^{-1}_{per})$ and a regularization $A^\eps_\infty \in \calC^\infty([0,T] \times [0,1])$ such that $A_\infty^\eps(t) \to A_\infty(t)$,  $t$ a.e. strongly in $\H^{-1}_{per}$. We then rewrite the integral term of \fref{freeL} as
\be \label{sumL}
\int_0^t \calL^\eps_1(s,\varphi,\psi)ds+\int_0^t \calL^\eps_2(s,\varphi,\psi)ds
\ee
with
\bee
 \calL^\eps_1(s,\varphi,\psi)&=&\big(\varrho_\infty(s) \varphi,H_{A^\eps_\infty(s)} \psi\big)-\big(H_{A^\eps_\infty(s)} \varphi, \varrho_\infty(s) \psi\big)\\
&=& Q_{A_\infty^\eps(s)}(\varrho_\infty(s) \varphi,\psi)-Q_{A_\infty^\eps(s)}(\varphi,\varrho_\infty(s)  \psi)\\
 \calL^\eps_2(s,\varphi,\psi)&=&\big(A^\eps_\infty(s) \varphi, \varrho_\infty(s) \psi\big)-\big(\varrho_\infty(s) \varphi,A_\infty^\eps(s) \psi\big).
\eee
Above, $H_{A_{\infty}^\eps(s)}=H+A_{\infty}^\eps(s)$ and $Q_{A^\eps_\infty(s)}$ is as in \fref{qquad}. Notice that $\calL_1^\eps$ and $\calL_2^\eps$ are well-defined since we have seen in \fref{rhoD} that $\varrho_\infty(s) \varphi \in \H^1_{per}$ for every $\varphi \in L^2$. 

We now pass to the limit in \fref{sumL}. With a similar estimate as in \fref{convQ}, we obtain, $s$ a.e., 
$$ \calL^\eps_1(s,\varphi,\psi) \to \calL_1(s,\varphi,\psi)=Q_{A_\infty(s)}(\varrho_\infty(s) \varphi,\psi)-Q_{A_\infty(s)}(\varphi,\varrho_\infty(s)  \psi).$$
In order to pass to the limit in the first integral of \fref{sumL}, we derive the following estimate, obtained from \fref{rhoH1} and the embedding $\H^1 \subset L^\infty$,
\bee
|\calL^\eps_1(s,\varphi,\psi)| &\leq&  \|\nabla (\varrho_\infty(s) \varphi)\|_{L^2}\|\nabla \psi\|_{L^2}+ \|A_\infty^\eps(s)\|_{\H^{-1}_{per}}\| (\varrho_\infty(s) \varphi) \psi\|_{\H^1}\\
&&+ \|\nabla (\varrho_\infty(s) \psi)\|_{L^2}\|\nabla \varphi\|_{L^2}+ \|A_\infty^\eps(s)\|_{\H^{-1}_{per}}\| (\varrho_\infty(s) \psi) \varphi\|_{\H^1}\\
&\leq & C.
\eee
Hence, for all $t\in [0,T]$, the first term in \fref{sumL} converges by dominated convergence to
$$
\int_0^t \left[Q_{A_\infty(s)}(\varrho_\infty(s) \varphi,\psi)-Q_{A_\infty(s)}(\varphi,\varrho_\infty(s)  \psi) \right]ds.
$$
Then, since $\varrho_\infty(t)$ and $H_{A_\infty(t)}$ commute, we have $(H_{A_\infty(s)} \varrho_\infty(s))^*=H_{A_\infty(s)} \varrho_\infty(s)$, and according to \fref{relQH} and \fref{rhoD}, the integrand becomes
\begin{align*}
&\big(H_{A_\infty(s)} \varrho_\infty(s) \varphi,\psi\big)-\big(\varphi, H_{A_\infty(s)} \varrho_\infty(s)  \psi \big)\\
&\hspace{2cm}=\big(H_{A_\infty(s)} \varrho_\infty(s) \varphi,\psi\big)-\big((H_{A_\infty(s)} \varrho_\infty(s ) )^*\varphi,   \psi \big)\\
&\hspace{2cm}=\big(H_{A_\infty(s)} \varrho_\infty(s) \varphi,\psi\big)-\big(H_{A_\infty(s)} \varrho_\infty(s) \varphi,   \psi \big)=0.
\end{align*}
The term involving $\calL_1^\eps$ therefore vanishes. It remains to treat $\calL_2^\eps$, that we recast as follows
$$
\calL^\eps_2(s,\varphi,\psi)=\big(A^\eps_\infty(s), \overline{\varphi}\, \varrho_\infty(s) \psi\big)_{\H^{-1}_{{per}},\H^{1}_{{per}}}-\big(A^\eps_\infty(s), \psi\, \overline{\varrho_\infty(s) \varphi} \big)_{\H^{-1}_{{per}},\H^{1}_{{per}}}.
$$
In a similar way as for the term $\calL^\eps_1$, we can pass to the limit in second term of \fref{sumL}, and obtain as limit
$$
\int_0^t \left[\big(A_\infty(s), \overline{\varphi}\, \varrho_\infty(s) \psi\big)_{\H^{-1}_{\textrm{per}},\H^{1}_{{per}}}-\big(A_\infty(s), \psi\, \overline{\varrho_\infty(s) \varphi} \big)_{\H^{-1}_{{per}},\H^{1}_{{per}}}\right]ds.
$$
Finally, since $\varrho_\infty \in L^\infty((0,T),\calE)$, the series
$$
\varrho_\infty^N(s)\varphi:=\sum_{p=1}^N \rho_p(s) \phi_p(s) (\phi_p(s),\varphi), \qquad \forall \varphi \in L^2,
$$
converges absolutely in $\H^1$, $s$ a.e.,  and we rewrite the integrand as 
\bee
\calL_2(s,\varphi,\psi)&=&\sum_{p \in \NN^*} \rho_p(s)\big(A_\infty(s), \overline{\varphi} \phi_p(s)\big)_{\H^{-1}_{{per}},\H^{1}_{{per}}} (\phi_p(s),\psi)\\
&&-\sum_{p \in \NN^*} \rho_p(s)\big(A_\infty(s), \overline{\phi_p(s)} \psi \big)_{\H^{-1}_{{per}},\H^{1}_{{per}}} \overline{(\phi_p(s),\varphi)}.
\eee
We therefore obtain the following equation for $\varrho_\infty$: for all $\varphi, \psi \in \calC^\infty_{{per}}([0,1])$, and all $t \in [0,T]$,
\be \label{reprhoinf}
(\varrho_\infty(t)\varphi,\psi)=(\varrho_\infty(0)\varphi,\psi)- i\int_0^t \calL_2(s,\varphi,\psi) ds.
\ee
Note that the above relation can be extended to $\varphi, \psi \in \H^1_{per}$ since $\calL_2$ verifies
$$
|\calL_2(s,\varphi,\psi)| \leq C \|\varphi\|_{\H^1_{per}} \|\psi\|_{\H^1_{per}}.
$$

\subsection*{Step 2: $n[\varrho_\infty(t)]$ is constant in time for $t \in [0,T]$} We show now that the local density of the second term of the r.h.s. of \fref{reprhoinf} is zero. This requires some attention since we do not know at this point that $A_\infty \varrho_\infty \in \calJ_1$. For this matter, we remark first that $\varrho_\infty(t) \in \calJ_1$ for all $t\in [0,T]$, and therefore that it is enough to identify $n[\varrho_\infty(t)]$ by choosing a particular basis of $L^2$ in the computation of $\Tr (\varrho_\infty(t) \phi)$, where $\phi \in \calC^\infty_{per}([0,1])$. We pick $(e_j)_{j \in \NN^*}$ the eigenbasis of $H$, for the simple reason that $e_j \in \H^1_{per}$. Then,
\be \label{sumN}
\sum_{j=1}^N (\varrho_\infty(t)\phi e_j,e_j)=\sum_{j=1}^N (\varrho_\infty(0)\phi e_j,e_j)-i\int_0^t \sum_{j=1}^N \calL_2(s,\phi e_j,e_j) ds.
\ee
Denote by $u_{N,p}$ and $v_{N,p}$ the partial sums
$$
u_{N,p}=\sum_{j=1}^N e_j \overline{(\phi_p, e_j)}, \qquad v_{N,p}=\sum_{j=1}^N e_j \overline{(\phi_p, e_j \phi)}.
$$
Since $\phi \in \calC^\infty_{per}([0,1])$, and $\phi_p(s) \in \H^1_{per}$, $s$ a.e. (to see this fact, choose $\varphi=\phi_p$ in \fref{rhoD} so that $\rho_p \phi_p \in \H^1_{per}$ , and use the fact that $\varrho_\infty$ is a quantum Maxwellian and has therefore a full rank according to Theorem \ref{theo1}, which implies that $\rho_p>0$ for all $p \in \NN^*$), $u_{N, p}$ and $v_{N,p}$ both converge strongly in $\H^1$, $s$ a.e., as $N \to \infty$ to $\phi_p$ and $\phi_p \overline{\phi}$, respectively. We have then
\bee
\sum_{j=1}^N \calL_2(s,\phi e_j,e_j)&=&\sum_{p \in \NN^*} \rho_p(s)\big(A_\infty(s), \overline{ \phi u_{N,p}} \phi_p(s)\big)_{\H^{-1}_{{per}},\H^{1}_{{per}}} \\
&&-\sum_{p \in \NN^*} \rho_p(s)\big(A_\infty(s), \overline{\phi_p(s)} v_{N,p} \big)_{\H^{-1}_{{per}},\H^{1}_{{per}}}\\
&:=& \sum_{p \in \NN^*} \big(U_{N,p}(s)-V_{N,p}(s) \big),
\eee
with, $s$ a.e., for all $p \in \NN^*$,
\begin{align*}
&\lim_{N \to \infty} U_{N,p}(s)=\rho_p(s) \big(A_\infty(s), \overline{ \phi} |\phi_p(s)|^2 \big)_{\H^{-1}_{{per}},\H^{1}_{{per}}},\\
& \lim_{N \to \infty} V_{N,p}(s)=\rho_p(s) \big(A_\infty(s), \overline{ \phi} |\phi_p(s)|^2 \big)_{\H^{-1}_{{per}},\H^{1}_{{per}}}.
\end{align*}
Since
$$
\|u_{N,p}\|_{\H^1} \leq \|\phi_p\|_{\H^1},
$$
we have the estimate, using that $\H^1 \subset L^\infty$,
\bea \nonumber
|U_{N,p}| &\leq& C \rho_p \|A_\infty\|_{L^\infty((0,T),\H^{-1}_{per})} \|\phi_p\|_{\H^1} \|u_{N,p}\|_{\H^1} \|\phi \|_{\H^1} \\
&\leq& C \rho_p \|\phi_p\|^2_{\H^1}, \label{UU}
\eea
and a similar one holds for $V_{N,p}$. Since moreover, 
$$\sum_{p  \in \NN^*}  \rho_p(s) \|\phi_p(s)\|^2_{\H^1} \leq  \|\varrho_\infty\|_{L^\infty((0,T),\calE)} < \infty, $$
 we can use dominated convergence for series and obtain that, $s$ a.e.,
$$
\lim_{N \to \infty}\sum_{j=1}^N \calL_2(s,\phi e_j,e_j)=0.
$$
Since moreover, thanks to \fref{UU},
$$
\left| \sum_{j=1}^N \calL_2(s,\phi e_j,e_j) \right| \leq C \|A_\infty\|_{L^\infty((0,T),\H^{-1}_{per})}  \|\varrho_\infty\|_{L^\infty((0,T),\calE)},
$$
and also $\varrho_\infty(t) \in \calJ_1$ for all $t \in [0,T]$, we can take the limit $N\to \infty$ in \fref{sumN} and obtain, for all $t \in [0,T]$ and all $\phi \in \calC^\infty_{{per}}([0,1])$,
\bee
(\phi,n[\varrho_\infty(t)])&=&(\phi,n[\varrho_\infty(0)]).
\eee
This means that $n[\varrho_\infty(t)])=n[\varrho_\infty(0)]$, for all $t\in [0,T]$.

\subsection*{Step 3: conclusion} Since $n[\varrho_\infty(t)])=n[\varrho_\infty(0)]$, for all $t\in [0,T]$, we have that $\varrho_e[\varrho_\infty(t)]=\exp(-H_{A_\infty(t)})$ is constant a.e. in $[0,T]$, say equal to $\exp(-H_{B_\infty})$ with $B_\infty \in \H^{-1}_{{per}} $. As a consequence, since  $\varrho_\infty(t)=\varrho_e[\varrho_\infty(t)]$ a.e. in $[0,T]$, we have that $\varrho_\infty(t)$ is constant a.e. in $[0,T]$. Since actually $\varrho_\infty$ is continuous with values in $\calJ_1$, this means that $\varrho_\infty(t)$ is constant for all $t\in[0,T]$ and equal to  $\exp(-H_{B_\infty})$. Therefore, \fref{freeL} yields, for every $\varphi,\psi \in \calC^\infty_{{per}}([0,1])$, $t \in [0,T]$,
$$
(\varrho_\infty \varphi,H \psi)=(H \varphi, \varrho_\infty \psi).
$$
Denoting by $(\mu_p, e_p)_{p \in \NN^*}$ the spectral elements of $H$ with domain $D(H)$, with in particular $e_p \in \calC^\infty_{{per}}([0,1])$, and where $e_1(x)=1$ is the ground state associated with $\mu_1=0$, we find, for any $p \neq 1$,
$$
0=(\varrho_\infty e_p,H e_1)=(H e_p, \varrho_\infty e_1)=\mu_p (e_p,\varrho_\infty e_1).
$$
Since $\mu_p>0$ for $p \neq 1$, this shows that $(e_p,\varrho_\infty e_1)=0$ for all $p \neq 1$. Since morever $\varrho_\infty=\exp(-H_{B_\infty})$ is a quantum Maxwellian and has therefore a full rank according to Theorem \ref{theo1}, we have $\ker \varrho_\infty=\{0 \}$ which ensures that $\varrho_\infty e_1$ is not identically zero. Together with the fact that the ground state is nondegenerate, this means that $\varrho_\infty e_1=c  e_1$, for some constant $c$. We have therefore obtained that $e_1$ is an eigenfunction of the operator $H_{B_\infty}$, say associated with the eigenvalue $\lambda_{q_0}$.  Hence, denoting by $(\lambda_p,\phi_p)_{p \in \NN^*}$ the spectral elements of $H_{B_\infty}$, we have, for every $\varphi \in \H^1_{{per}}$,
\bee
Q_{B_\infty}(\varphi,e_1)&=&(\varphi,H_{B_\infty} e_1)=\lambda_{q_0} (\varphi,e_1)\\
&=& (\nabla \varphi, \nabla e_1)+(B_\infty,\overline{\varphi} e_1)_{\H^{-1}_{{per}},\H^1_{{per}}}=(B_\infty,\overline{\varphi} e_1)_{\H^{-1}_{{per}},\H^1_{{per}}}.
\eee
Since $\varphi$ is arbitrary, this shows that $B_\infty=\lambda_{q_0}$, and $B_\infty$ is therefore independent of $x$. The operator $\varrho_\infty$ then reads $e^{-\lambda_{q_0}} e^{-H}$, and verifies, for $t\in [0,T]$,
$$
e^{-\lambda_{q_0}} \Tr e^{-H}=\Tr \varrho_\infty= \Tr \varrho^0.
$$
Following Lemma \ref{momN}, this means that $\varrho_\infty(t)=\varrho_g$ for all $t \in [0,T]$, and in particular that the subsequence $(\varrho_k)_k$ verifies $\varrho(t_k) \to \varrho_g$ in $\calJ_1$. Since $\varrho_g$ is the unique solution to the global moment problem, the entire sequence converges, which concludes the proof of the theorem.
\section{Additional Proofs} \label{secother}
\subsection{Proof of Lemma \ref{eqfreeE}} \label{prooffreeE} 
We start by regularizing $F$ in order to justify the calculations. For $\eta \in (0,1]$, let  $\beta_\eta(x)=(x+\eta) \log (x+\eta)-x-\eta \log \eta$, and define for $\varrho \in \calE_+$,
\bee
F_\eta(\varrho)&=&\Tr \big(\beta_\eta(\varrho)\big)+\Tr \left(\sqrt{H} (\II+\eta\sqrt{H})^{-1} \varrho (\II+\eta\sqrt{H})^{-1} \sqrt{H} \right)\\
&:=& S_\eta(\varrho)+ K_\eta(\varrho).
\eee
For any trace-class self-adjoint operator $u$, the G\^ateaux derivative of $S_\eta$ at $\varrho \in \calE_+$ in the direction $u$ is well-defined and given by, according to Lemma \ref{lemdiff} of the Appendix,
$$
DS_\eta(\varrho)(u)=\Tr \big(\log(\eta+\varrho) u\big).
$$
Hence, since $\partial_t \varrho \in \calC^0([0,T],\calJ_1)$, we can write
$$
\frac{d S_\eta(\varrho(t))}{dt}=DS_\eta(\varrho)(\partial_t \varrho(t))=\Tr\big(\log(\eta+\varrho(t)) \partial_t \varrho(t)\big).
$$
Note that the expression makes sense since $\log(\eta+\varrho(t))$ is a bounded operator. Similarly, since the operator $\sqrt{H} (\II+\eta\sqrt{H})^{-1}$ is bounded, 
$$
\frac{d K_\eta(\varrho(t))}{dt}=\Tr \left(\sqrt{H} (\II+\eta\sqrt{H})^{-1} \partial_t \varrho(t) (\II+\eta\sqrt{H})^{-1} \sqrt{H} \right).
$$
We can then replace $\partial_t \varrho$ by its expression given by the Liouville equation. Now we claim that, since $\varrho$ and $\log(\eta+\varrho)$ commute,
\be \label{trlogzero}
\Tr\big(\log(\eta+\varrho(t)) [H, \varrho(t)]\big)=0, \qquad \forall t \in [0,T].
\ee
Even though $H \varrho(t) \in \calJ_1$ for each $t$ positive, the proof requires some regularization since $H$ is unbounded and cyclicity of the trace cannot be directly applied. Let thus $H_\eps=H(\II+\eps H)^{-1} \in \calL(L^2)$, and \fref{trlogzero} is verified for $H$ replaced by $H_\eps$ by cyclicity of the trace. Write then,
\begin{align*}
&\left| \Tr \big(\log(\eta+\varrho(t)) (H\varrho(t)-H_\eps \varrho(t))\big) \right|\\
 & \qquad \qquad\leq \| \log(\eta+\varrho(t)) \|_{\calL(L^2)} \| H \varrho(t)\|_{\calJ_1} \| \II-(\II+\eps H)^{-1}\|_{\calL(L^2)}.
\end{align*}
Since the first two terms on the r.h.s are uniformly bounded in time, and since the last one converges to zero as $\eps \to 0$, we conclude, along with a similar argument for the other term in the commutator, that the claim \fref{trlogzero} holds. Therefore,
$$
\frac{d S_\eta(\varrho(t))}{dt}:=S'_\eta(t):=\frac{1}{\tau} \Tr\big(\log(\eta+\varrho(t)) (\varrho_e(t)-\varrho(t))\big).
$$
Furthermore, proceeding as for \fref{trlogzero}, we have that
$$
\Tr \left(\sqrt{H} (\II+\eta\sqrt{H})^{-1} [H, \varrho] (\II+\eta\sqrt{H})^{-1} \sqrt{H} \right)=0,
$$ 
and consequently
$$
\frac{d K_\eta(\varrho(t))}{dt}=K'_\eta(t):=\frac{1}{\tau} \Tr\left(\sqrt{H} (\II+\eta\sqrt{H})^{-1} (\varrho_e(t)-\varrho(t))(\II+\eta\sqrt{H})^{-1} \sqrt{H} \right).
$$
For any $t,s \geq 0$, write then
\begin{equation}
\label{F_eta}
F_\eta(\varrho(t))-F_\eta(\varrho(s))=\int_s^t (S'_\eta(z)+K'_\eta(z))dz.
\end{equation}
We pass now to the limit in \eqref{F_eta}. Let 
$$
K'(t)=\frac{1}{\tau}\Tr \big(H (\varrho_e(t)-\varrho(t))\big),
$$
which is well defined since both $H\varrho_e$ and $H\varrho$ belong to $L^\infty((0,T),\calJ_1)$. Hence,
\begin{align*}
&\tau |K'_\eta(t)-K'(t)|\\
&\qquad \qquad= \left|\Tr\left(((\II+\eta\sqrt{H})^{-1} H (\II+\eta\sqrt{H})^{-1} -H) (\varrho_e(t)-\varrho(t)) \right) \right|\\
&\qquad \qquad\leq \left\| \left((\II+\eta\sqrt{H})^{-1} H (\II+\eta\sqrt{H})^{-1} -H\right) (\II+H)^{-1}  \right\|_{\calL(L^2)} \\
&\qquad \qquad \qquad \times \|(\II+H)(\varrho_e(t)-\varrho(t)) \|_{\calJ_1}.
\end{align*}
Since
$$
\| \big((\II+\eta\sqrt{H})^{-1} H (\II+\eta\sqrt{H})^{-1}-H\big) (\II+ H)^{-1}  \|_{\calL(L^2)} \to 0 \qquad \textrm{as} \qquad \eta \to 0,
$$
 and both $H\varrho_e$ and $H\varrho$ are bounded in $L^\infty((0,T),\calJ_1)$, we have that 
$$\lim_{\eta\to 0} \|K'_\eta-K'\|_{L^\infty(0,T)}=0.$$
Recall now that $\varrho_e(t)=e^{-(H+A(t))}$ with $A\in L^\infty((0,T),L^2(0,1))$.
Let $A_\eps \in \calC^0([0,T] \times [0,1])$ be a smooth sequence such that $A_\eps \to A$  strongly in $L^\infty((0,T),L^2(0,1))$. Then,
$$
K'(t)=\frac{1}{\tau}\Tr \big((H+A_\eps(t)) (\varrho_e(t)-\varrho(t))\big)-\frac{1}{\tau}\Tr \big(A_\eps(t)(\varrho_e(t)-\varrho(t))\big).
$$
The second term above is equal to zero since $n[\varrho_e(t)]=n[\varrho(t)]$ for all $t \geq 0$. We pass to the limit in the first one observing that
\begin{align*}
&\left|\Tr \big((A_\eps(t)-A(t)) (\varrho_e(t)-\varrho(t))\big) \right| \\
& \qquad \leq \|(A_\eps(t)-A(t)) (\II+H)^{-1} \|_{\calL(L^2)} \|(\II+H) (\varrho_e(t)-\varrho(t)) \|_{\calJ_1},
\end{align*}
together with the facts that $H\varrho, H\varrho_e \in L^\infty((0,T), \calJ_1)$ and
\bee
\|(A_\eps(t)-A(t)) (I+H)^{-1} \|_{\calL(L^2)} &\leq & C \| A_\eps(t)-A(t) \|_{L^2}  
\eee
since $(\II+H)^{-1} : L^2 \to \H^2 \subset L^\infty$. Therefore, we find that
$$
K'(t)=-\frac{1}{\tau}\Tr \big(\log (\varrho_e(t)) (\varrho_e(t)-\varrho(t))\big).
$$
We are done with the term $K_\eta'$. 

Let us now prove that that the free energy $F_\eta(\varrho)$ converges to $F(\varrho)$ uniformly on $[0,T]$ since $\varrho \in \calC^0([0,T],\calE_+)$. The convergence of the energy term is straightforward since $\varrho \in \calC^0([0,T],\calE_+)$. Regarding the entropy term, we remark first that the function $\beta_\eta$ converges to $\beta$ uniformly on all $[0,M]$, $M>0$, and that one has
$$\forall s\in [0,M],\quad \left|\beta_\eta(s)\right|\leq C_M\sqrt{s},$$
with $C_M$ independent of $\eta$. Let $M=\max_{t \in [0,T]}\|\varrho(t)\|_{\calL(L^2)}$. For all $N\in \NN^*$, by using Lemma \ref{lieb}, we get, for all $t \in [0,T]$,
\bea \nonumber
\sum_{p\geq N}\left|\beta_\eta(\rho_p(t))\right|&\leq& C_M\sum_{p\geq N}\sqrt{\rho_p(t)} \leq  C_M \left(\sum_{p\geq N} p^2\rho_p(t) \right)^{1/2} \left(\sum_{p\geq N} \frac{1}{p^2}\right)^{1/2} \\
&\leq& \frac{C_M}{\sqrt{N}}\left(\Tr \sqrt{H}\varrho(t)\sqrt{H}\right)^{1/2} \leq \frac{C_M}{\sqrt{N}}, \label{estbetaeta}
\eea
where $\rho_p(t)$ denotes the $p$-th eigenvalue of $\varrho(t)$. Hence, decomposing
\bee
\left|\Tr \big(\beta_\eta(\varrho(t)) \big) -\Tr \big(\beta(\varrho(t)) \big)\right|&\leq&
\sum_{p< N}\left|\beta_\eta(\rho_p(t))-\beta(\varrho_p(t))\right|\\
&&+\sum_{p\geq N}\left|\beta_\eta(\rho_p(t))\right|+\sum_{p\geq N}\left|\beta(\rho_p(t))\right|,
\eee
we obtain the desired result from the uniform convergence of $\beta_\eta$, the fact that $\rho_p(t) \leq M$, and finally \fref{estbetaeta}. 

Let us now prove the convergence of the term $S'_\eta$ in \eqref{F_eta}.
One shows in the same way that, that uniformly on $[0,T]$, 
$$
\Tr\big(\log(\eta+\varrho(t)) \varrho(t)\big) \to \Tr\big(\log(\varrho(t)) \varrho(t)\big) \qquad \textrm{as} \qquad \eta \to 0.
$$
It remains  finally to treat the term $S'_{2,\eta}:=\Tr (\log (\eta+\varrho(t))\varrho_e(t))$, which is a little more technical since we cannot simultaneously diagonalize $\varrho$ and $\varrho_e$. Denote by $(\rho_i,\phi_i)_{i \in \NN^*}$ the spectral elements of $\varrho$. Let $\eta_0 \in (0,1)$, $\eta \in (0,\eta_0)$, and $N(t) \in \NN^*$ such that $\eta_0+\rho_{N(t)+1}(t) <1$ and $\eta_0+\rho_{N(t)}(t) \geq 1$. Then, 
\bee
\tau S'_{2,\eta}(t)&=&\sum_{n\leq N(t)} \log(\eta+\rho_n(t)) (\varrho_e(t) \phi_n,\phi_n)+\sum_{n>N(t)}^\infty \log(\eta+\rho_n(t)) (\varrho_e(t) \phi_n,\phi_n)\\
&:=&f_\eta^1(t)+f_\eta^2(t).
\eee
Since $\varrho_e \geq 0$ and $\varrho_e \in \calC^0([0,T],\calJ_1)$, we have
$$
|f_\eta^1(t)| \leq M_0 \sum_{n\leq N(t)} (\varrho_e(t) \phi_n,\phi_n) \leq M_0 \max_{t \in [0,T]}\| \varrho_e(t)\|_{\calJ_1} \leq C,
$$
where
$$
M_0=\max_{ x \in [1-\eta_0, \eta_0 + \max_{t \in [0,T] }\| \varrho(t)\|_{\calL(L^2)}]} |\log x|.
$$
We can then use dominated convergence to obtain that, for all $(t,s) \in [0,T] \times [0,T]$:
$$
\lim_{\eta \to 0} \int_{s}^t f_\eta^1(z) dz= \int_{s}^t \sum_{n\leq N(z)} \log(\rho_n(z)) (\varrho_e(z) \phi_n,\phi_n) dz.
$$
Regarding $f_\eta^2$, we observe that for all $t\in[0,T]$ and $n>N(t)$, the term $-\log(\eta+\rho_n(t)) (\varrho_e(t) \phi_n,\phi_n)$ is positive and strictly increasing as $\eta \to 0$. Using monotone convergence, first for series, we obtain that pointwise in $t$, $f_\eta^2(t)\to f^2(t)$, and secondly for integrals, we have for all $(t,s) \in [0,T] \times [0,T]$:
$$
\lim_{\eta \to 0} \int_{s}^t f_\eta^2(z) dz= \int_{s}^t \sum_{n>N(z)} \log(\rho_n(z)) (\varrho_e(z) \phi_n,\phi_n) dz.
$$ 
Hence, for all $(t,s) \in [0,T] \times [0,T]$,
$$
\tau \lim_{\eta \to 0} \int_s^t S'_\eta(z)dz= \int_s^t \Tr \big(\log (\varrho(z)) (\rho_e(z)-\varrho(z))\big)dz.
$$
This concludes the proof.
\subsection{Proof of Proposition \ref{severalestimates}} \label{proofseveral} The proof consists in combining various results of \cite{MP-JSP}. We start with the first estimate. Given $n$ as in Theorem \ref{theo1}, let $\phi_1:=\|n\|^{-1/2}_{L^1} \sqrt{n}$ and complete $\phi_1$ to an orthonormal basis  $(\phi_i)_{i \geq 1}$ of $L^2$. For all $\psi \in L^2$, consider the density operator $\nu$ defined by $\nu \psi:=\sqrt{n}\, (\sqrt{n},\psi)=\|n\|_{L^1} \phi_1(\phi_1,\psi).$ Since $n \in \Hunper$, it follows that $\nu \in \calE_+$. Besides, since $n[\nu]=n$ and $\varrho[n]$ is the minimizer of $F$, we have
\be \label{ineqF}
F(\varrho[n]) \leq F(\nu).
\ee
Since $\nu$ is of rank one, it it then not hard to see that 
$$
F(\nu)=\beta(\|n\|_{L^1})+\| \nabla \sqrt{n}\|^2_{L^2}.
$$
Furthermore, denoting by $(\rho_p)_{p> p_0}$ the eigenvalues of $\varrho[n]$ that belong to the interval $(0,e]$, we can show, proceeding as in \fref{estbetaeta}, that there exists $C>0$ such that
\be \label{estbeta}
-\Tr \big(\beta(\varrho[n]) \big) \leq \sum_{p > p_0}\left|\beta (\rho_p)\right| \leq C \left( \Tr \big(\sqrt{H} \varrho[n] \sqrt{H}\big)\right)^{1/2}.
\ee
Easy algebra then yields from \fref{ineqF} and \fref{estbeta},
$$
\Tr \big( \sqrt{H} \varrho[n] \sqrt{H} \big) \leq C\left(1+\beta(\|n\|_{L^1})+\| \nabla \sqrt{n}\|^2_{L^2}\right
).
$$
This provides us with an estimate for $\varrho$ in $\calE$. Regarding the entropy term in \fref{estimsolmom1}, we only need to estimate the sum of the eigenvalues of $|\varrho[n] \log \varrho[n]|$. We already have \fref{estbeta} for the indices $p>p_0$. For the remaining part $p \leq p_0$, owing to the fact that
$$\rho_p \leq \| \varrho[n]\|_{\calL(L^2)} \leq \|\varrho[n]\|_{\calJ_1}=\|n\|_{L^1}, \qquad \forall p \in \NN^*,$$
we find the estimate
\bee
\sum_{p \leq p_0}\left|\rho_p \log \rho_p\right|&\leq& \log \|n\|_{L^1} \Tr\big(\varrho[n] \big)=\log \|n\|_{L^1}\|n\|_{L^1}.
\eee
Together with \fref{estbeta}, this yields
$$
\sum_{p \geq 1}\left|\rho_p \log\rho_p\right| \leq C+C \beta(\|n\|_{L^1})+C\|\varrho[n]\|_{\calE},
$$
which concludes the proof of the first estimate \fref{estimsolmom1}. 

For the second estimate \fref{estimsolmom2}, we deduce from \fref{defAA} and \fref{ident} that, since $\H^1 \subset L^\infty$,
$$
\left|\left(A,\psi\right)_{\H^{-1}_{per},\Hunper}\right|\leq C\left\|\frac{\psi}{n}\right\|_{\H^1}\left(\Tr \big(|\varrho\log\varrho|\big)+\Tr \big(\varrho\big)+\Tr\big(\sqrt{H}\varrho\sqrt{H}\big)\right),
$$
and the result follows from
$$
\left\|\frac{\psi}{n}\right\|_{\H^1} \leq \frac{\| \psi\|_{\H^1}}{\underline{n}}+C\frac{\| \psi\|_{\H^1}\| \nabla \sqrt{n}\|_{L^2} }{\underline{n}^{3/2}}
$$
and easy manipulations. This ends the proof of the proposition.
\subsection{Proof of Lemma \ref{hold0}} \label{proofhold0}
Decompose $\calL(t) \varrho-\varrho$ as
\be \label{L1}
\calL(t)\varrho -\varrho=e^{-i tH} \varrho e^{i tH}- \varrho e^{i tH}+\varrho e^{i tH}-\varrho=S_1+S_2.
\ee
Since, for all $\lambda\in \RR^+$, we have
\be \label{L2} \left|\frac{e^{-it\lambda}-1}{1+\lambda}\right|=\left|\int_0^t\frac{\lambda e^{-is\lambda}}{1+\lambda}ds\right|\leq t,\ee
we deduce that
$$
\| (e^{-i tH}-\II)(\II+H)^{-1}\|_{\calL(L^2)} \leq t.
$$
Therefore,
$$
\|S_1\|_{\calJ_1} \leq \| (e^{-i tH}-\II)(\II+H)^{-1} (\II+H)\varrho e^{i tH}\|_{\calJ_1} \leq t \|(\II+H)\varrho \|_{\calJ_1}.
$$
In the same way,
$$
\|S_2\|_{\calJ_1}=\|S^*_2\|_{\calJ_1}=\| e^{-i tH}\varrho-\varrho\|_{\calJ_1} \leq t \|(\II+H)\varrho \|_{\calJ_1}.
$$
Now, using the polar decomposition $(\II+H)\varrho=U|(\II+H)\varrho|$, we conclude by noticing that
\bee
\Tr \big(|(\II+H)\varrho| \big)&=&\Tr \big( U^*(\II+H)\varrho(\II+H)(\II+H)^{-1} \big)\\
&\leq & \Tr \big((\II+H)\varrho(\II+H) \big)\\
&\leq & C \| \varrho\|_{\calH},
\eee
since $U$ is an isometry on $\overline{\mbox{Ran\,}(\II+H)\varrho}$ and $(\II+H)^{-1}$ is bounded. This ends the proof.
\subsection{Proof of estimate \eqref{hold00}} \label{proofhold00}
We start from \fref{L1} and write, using \fref{L2} with $\lambda \in \RR^+$,
$$
\left|\frac{e^{-it\lambda}-1}{1+\sqrt{\lambda}}\right|=\left|e^{-it\lambda}-1\right|^{1/2}\frac{|1+\lambda|^{1/2}}{1+\sqrt{\lambda}}\left|\frac{e^{-it\lambda}-1}{1+\lambda}\right|^{1/2} \leq  2\left|\frac{e^{-it\lambda}-1}{1+\lambda}\right|^{1/2}\leq 2 \sqrt{t}.  
$$
This yields
$$
\| (e^{-i tH}-\II)(\II+\sqrt{H})^{-1}\|_{\calL(L^2)} \leq 2 \sqrt{t},
$$
and we conclude following the same lines as in the proof of Lemma \ref{hold0}.
\subsection{Proof of Lemma \ref{hold2}} \label{proofhold2}
Let $w(t,s):=\varrho(t)-\varrho(s)$. Since $\varrho(t) \in \calH$ for all $t\in [0,T]$, we deduce that $|w(t,s)| \in \calH$, for all $(t,s)$, and that $ |w(t,s)|^{1/2} H \in \calJ_2 $. Define then
$$
S:=\Tr \big(\sqrt{H} |w(t,s)| \sqrt{H} \big)=\Tr  \big(|w(t,s)| H \big).
$$
We justify the last equality by writing
\bee
S&=&\lim_{\eps \to 0}\Tr \big(\sqrt{H}(\II+\eps \sqrt{H})^{-1}|w(t,s)|\sqrt{H}\big)\\
&=&\lim_{\eps \to 0}\Tr \big(|w(t,s)|H(\II+\eps \sqrt{H})^{-1}\big)\\
&=&\Tr \big(|w(t,s)|H \big).
\eee
The first equality holds since $ \sqrt{H}(\II+\eps \sqrt{H})^{-1} (\II+\sqrt{H})^{-1}$ converges to $\sqrt{H} (\II+\sqrt{H})^{-1}$ in $\calL(L^2)$ and since $  (\II+\sqrt{H}) |w(t,s)|\sqrt{H}\in \calJ_1$. The last equality follows from the convergence of $(\II+\eps \sqrt{H})^{-1}$ in $\calL(L^2)$ to the identity, and from the fact that $|w(t,s)|H \in \calJ_1$. Then,
\bee
S&=&\Tr  \big(|w(t,s)|^{1/2}|w(t,s)|^{1/2} H \big)\\
&\leq&  \||w(t,s)|^{1/2}\|_{\calJ_2} \| |w(t,s)|^{1/2} H\|_{\calJ_2}\\
& \leq & \big(\Tr \big(|w(t,s)| \big)\big)^{1/2} \left( \Tr \big(H |w(t,s)| H \big)\right)^{1/2} \\
&\leq & \big(\Tr \big(|w(t,s)|\big)\big)^{1/2} \sqrt{2} \|\varrho\|^{1/2}_{\calC^0([0,T],\calH)}.
\eee
This ends the proof.

\section{Appendix} \label{appen}
The next four lemmas are proved in \cite{MP-JSP}.
\begin{lemma}[\cite{MP-JSP}, Lemma 3.1]\label{strongconv} 
Let $(\varrho_k)_{k\in \NN}$ be  a bounded sequence of $\calE_+$. Then, up to an extraction of a subsequence, there exists $\varrho\in \calE_+$ such that
$$
\varrho_k\to\varrho\mbox{ in }\calJ_1\quad \mbox{and}\quad \sqrt{\varrho_k}\to \sqrt{\varrho}\mbox{ in }\calJ_2\quad \mbox{as } k\to +\infty
$$
and
$$
\Tr \big(\sqrt{H}\varrho\sqrt{H}\big)\leq \liminf_{k\to +\infty} \Tr \big(\sqrt{H}\varrho_k\sqrt{H}\big).
$$
\end{lemma}
\begin{lemma}[\cite{MP-JSP}, Lemma 5.2]
\label{propentropie}
The application $\varrho\mapsto \Tr (\varrho \log \varrho-\varrho)$ possesses the following properties.\\
(i) There exists a constant $C>0$ such that, for all $\varrho\in \calE_+$, we have
\be
\label{souslin}
\Tr \big(\varrho\log \varrho-\varrho\big)\geq -C\left(\Tr \big(\sqrt{H}\varrho\sqrt{H}\big)\right)^{1/2}.
\ee
(ii) Let $\varrho_k$ be a bounded sequence of $\calE_+$ such that $\varrho_k$ converges to $\varrho$ in $\calJ_1$, then $\varrho_k \log \varrho_k-\varrho_k$ converges to $\varrho \log \varrho-\varrho$ in $\calJ_1$.\\
\end{lemma}
\begin{lemma}[\cite{MP-JSP}, Lemma 5.3] \label{lemdiff} For $\eta \in (0,1]$, let  $\beta_\eta(x)=(x+\eta) \log (x+\eta)-x-\eta \log \eta$, and let $\varrho \in \calE_+$ and $\omega$ be a trace-class self-adjoint operator. Then, the G\^ateaux derivative of the application 
$$\varrho\mapsto S_\eta(\varrho)=\Tr \big(\beta_\eta(\varrho) \big)$$ 
at $\varrho$ in the direction $\omega$ is well-defined and we have
$$
D S_\eta(\varrho) (\omega) = \Tr \big(\beta_\eta'( \varrho) \omega\big). 
$$ 
\end{lemma}
\begin{lemma} [\cite{MP-JSP}, Lemma A.1] \label{lieb} Let $\varrho \in \calE_+$ and denote by $(\rho_p)_{p\geq 1}$ the nonincreasing sequence of nonzero eigenvalues of $\varrho$, associated to the orthonormal family of eigenfunctions $(\phi_p)_{p\geq 1}$. Denote by $(\lambda_p[H])_{p \geq 1}$ the nondecreasing sequence of nonzero eigenvalues of the Hamiltonian $H$. Then we have
$$
\Tr \big(\sqrt{H} \varrho \sqrt{H}\big)=\sum_{p \geq 1} \rho_p \,(  \sqrt{H} \phi_p, \sqrt{H}\phi_p) \geq \sum_{p \geq 1} \rho_p \,\lambda_p[H].
$$
As a consequence,
$$
\Tr \big( \varrho ^{2/3}\big) \leq C \left(\Tr \big( \sqrt{H} \varrho \sqrt{H}\big)\right)^{2/3}.
$$
\end{lemma}

Only the last estimate above is not proved in \cite{MP-JSP}, it follows from
$$
\Tr \big(\varrho^{2/3}\big)=\sum_{p \geq 1} \rho_p^{2/3} \leq \left(\sum_{p \geq 1}\rho_p \lambda_p[H]\right)^{2/3} \left( \sum_{p \geq 1}  (\lambda_p[H])^{-2}\right)^{1/3},
$$
and the fact that $\lambda_p[H]=( 2\pi p)^2$.
\bs

\noindent  
\textbf{Acknowledgment.} O. Pinaud acknowledges support from NSF CAREER grant DMS-1452349 and F. M\'ehats acknowledge support from the ANR project Moonrise ANR-14-CE23-0007-01.


\end{document}